\documentclass[11pt,a4paper]{article}

\usepackage{amsfonts,amssymb,amsmath,amsthm,amscd}
\usepackage{mathtools}
\usepackage{empheq}
\usepackage{cases}
\usepackage{latexsym}

\usepackage{graphicx}
\usepackage{float}
\usepackage[margin=1in]{geometry}

\usepackage{fancyhdr}
\usepackage{mdframed}
\usepackage{enumitem}

\usepackage{listings}
\usepackage{algorithm}
\usepackage{algorithmic}
\usepackage{hyperref}
\usepackage{xcolor}

\usepackage{comment}
\usepackage{appendix}
\usepackage{mathrsfs}
\usepackage{stmaryrd}
\usepackage{bbm}
\hypersetup{
    colorlinks=true,
    linkcolor=blue,
    filecolor=magenta,      
    urlcolor=cyan,
    citecolor=green,
}

\def\p{{\mathbb P}}
\def\e{{\mathbb E}}

\newtheorem{theorem}{Theorem}[section] \newtheorem{lemma}[theorem]{Lemma}
\newtheorem{proposition}[theorem]{Proposition}

\newtheorem{conjecture}[theorem]{Conjecture}

\newtheorem{remark}[theorem]{Remark}

\newcommand{\dega}[2]{\deg_{V_{#1}}(#2)}
\newcommand{\degb}[3]{\deg_{V_{#1}}(#2,#3)}
\newcommand{\degc}[2]{d_{#1}(#2)}
\newcommand{\V}[1]{V_{#1}}
\newcommand{\gV}[1]{\mathbf{V}^{#1}}
\newcommand{\pd}[3]{p_{#1}(\V{#2},\V{#3})}
\newcommand{\wda}[3]{W_{#1}(#2,#3)}
\newcommand{\wdb}[2]{W_{#1}(#2)}
\numberwithin{equation}{section}

\makeatletter
\newcommand{\subjclass}[2][1991]{%
  \let\@oldtitle\@title%
  \gdef\@title{\@oldtitle\footnotetext{#1 \emph{Mathematics subject classification.} #2}}%
}
\newcommand{\keywords}[1]{%
  \let\@@oldtitle\@title%
  \gdef\@title{\@@oldtitle\footnotetext{\emph{Key words and phrases.} #1.}}%
}
\makeatother

\title{Structural Properties of the Geometric Preferential Attachment Model}
\author{Chenxu Feng\\Peking University\and Yifan Li\\New York University}

\begin{document}

\maketitle

\begin{abstract}
    This paper analyzes key properties of networks generated by geometric preferential attachment. We establish that the expected number of triangles is proportional to that of the standard preferential attachment model, with a proportionality constant equal to the ratio of the number of triangles between a random geometric graph and an Erd\H{o}s-R\'enyi graph. Furthermore, we prove that the maximum degree grows polynomially with the network size, sharing the same exponent as the standard model; however, the spatial constraint induces a slower growth rate in the network's early evolution. Finally, we extend prior results on connectivity and diameter to the case of networks with finite out-degrees.
\end{abstract}
\section{Introduction}

The formation mechanisms of real-world complex networks—such as the Internet, the World Wide Web, and social networks—have been a subject of intense study. Seminal work, including \cite{Albert1999,Faloutsos1999,Jeong2000}, established that these networks commonly exhibit scale-free degree distributions and small-world properties. The preferential attachment model (PAM) \cite{BARABASI1999173} provides an elegant generative explanation for these two phenomena. However, the classical PAM fails to capture other ubiquitous features of real networks, notably strong clustering (\cite{Watts1998}) and significant community structures (\cite{Girvan2002}). This shortcoming has motivated a wealth of model variants that incorporate additional mechanisms, though a unified model that perfectly captures all these features remains elusive.

In this paper, we focus on a geometric preferential attachment model (GPM) which was first introduced by Flaxman, Frieze, and Vera in \cite{Flaxman01012006} and \cite{Flaxman01012007}. In this model, each vertex is added sequentially uniformly on a $d$-dimensional spherical surface $S$, and the probability that the new vertex is connected to a previous vertex is based on the distance between them and the degree of the previous vertex. Combining the preferential attachment model with additional spatial structure, the GPM exhibits higher clustering coefficient, which is closer to the real world network. These characteristics of the GPM partially account for the missing properties in the PAM.

The GPM and its variants have been studied in recent decades. \cite{Zuev2015} constructed the GPM in hyperbolic space and compared it with the Internet network. \cite{Jordan2013} extended the power law of the degree sequence to the GPM in non-uniform metric space. \cite{Jacob2015} studied a slightly different model of the GPM. In their model, a new vertex connects to vertices uniformly and randomly selected from a candidate set, where an old vertex is included depending on its distance to the new vertex and its degree. They proved that the in-degree distribution follows a power law, while the out-degree of any node is bounded by a logarithmic function in their model.

However, other important parameters such as the maximum degree and diameter of the graph have not been fully studied. Moreover, although \cite{Flaxman01012006} and \cite{Flaxman01012007} obtained estimates for the connectivity threshold and diameter when the out-degree of each vertex is $\Omega(\log(n))$, establishing those estimates in the constant degree regime remains largely open. In this paper, we focus on the case where the out-degree is constant and provide estimates for some basic structural properties.

\subsection{Definition of the model}\label{model and result}
We first provide a rigorous definition of GPM. Our model is similar to the fitness model in \cite{Flaxman01012006}. Let $S$ be a $d$-dimensional spherical surface $S^d$ with a surface area of 1, and $D(x,y)$ be the Euclidean distance on $S$. Let $p\in (0,1)$ be a parameter and $r$ be the detection radius such that the area of the detection region $B(x,r)\cap S$ is $p$ for all $x\in S$.

The process begins with an empty graph $GPM_0$. To form $GPM_{n}$ from $GPM_{n-1}$, a new vertex $\V n$ is placed uniformly on $S$, resulting in $GPM_{n,0}$. Then $m$ edges $e_{n,1}$, $e_{n,2} \dots e_{n,m}$ are added sequentially. Let $GPM_{n,i}$ denote the graph after adding edges $e_{n,1}$, $e_{n,2} \dots, e_{n,i}$. The edge $e_{n,i+1}$ is then added with probability
\begin{align}
    \p(e_{n,i+1}=(V_n,V_k)\mid GPM_{n,i})=
    \begin{cases}
        \frac{\mathbf{1}_{\{D(\V n,\V k)\leq r\}}(\degb{k}{n}{i}+m\delta)}{\sum\limits_{D(V_n,V_j)\leq r}(\degb{j}{n}{i}+m\delta)+m-i}\quad & \text{if}\quad 1\leq k \leq n-1, \\
        \frac{\degb{n}{n}{i}+m\delta+m-i}{\sum\limits_{D(V_n,V_j)\leq r}(\degb{j}{n}{i}+m\delta)+m-i}\quad & \text{if}\quad k = n,
    \end{cases}\label{edgeconnectprob}
\end{align}
where $\degb{k}{n}{i}$ denotes the degree of vertex $\V k$ in $GPM_{n,i}$, and $\delta >-m$ is the fitness parameter. Here we allow self-loops (i.e., allow $k=n$) to prevent the denominator from being zero when no existing vertices lie within distance $r$ of the new vertex. It is straightforward to see that when $\delta$ grows larger, the distribution of new edges is more uniform and the rich-get-richer phenomenon is weaker. In this paper, we assume $\delta>0$. We define $GPM_n := GPM_{n,m}$.

\subsection{Notation}\label{section-notation}

In this section, we record a list of notations that we shall use throughout the paper. Recall that $S$ is the surface of a ball of dimensions $d$ with a surface area of $1$. Also, recall that $D(x,y)$ is the Euclidean distance on $S$. For any $X\in S$, let $A(X)$ denote the area of $X$. Recall that $GPM_{n,i}$ is the network generated by the geometric preferential attachment process after adding the first $i$ edges of the vertex $\V n$ and $GPM_n:=GPM_{n,m}$. 

We use $RGG_n$ to denote a random geometric graph on the same vertex set as $GPM_n$, and two vertices $\V i,\V j$ are connected by an edge if and only if $D(\V i,\V j)\leq r$. This random geometric graph represents the ``possible edges" in the geometric preferential attachment network. When a new vertex is added, its neighbors in RGG form the candidate set for the edge-adding process. 

Let $\gV n:=\{\V 1,\V 2\dots, \V n\}$ be the vertex set in $GPM_n$. Denote $p_n(x,y)$ as the graph distance between the vertices $x$ and $y$ in $GPM_n$ ($0$ when $\V i$ and $\V j$ are disconnected), $d_n(x,y)$ as the graph distance on the random geometric network with vertex set $\gV n\cup \{x,y\}$ (note that $p_n$ is defined on the vertex set $\gV n$, while $d_n$ is defined on $S$). Define $diam(GPM_n)=\max\limits_{1\leq i<j\leq n}(\pd{n}{i}{j})$ as the diameter of $GPM_n$.

Recall that $\degb{i}{n}{j}$ is the degree of $\V i$ in $GPM_{n,j}$. Define $\wda{i}{n}{j}:=\degb{i}{n}{j}+m\delta$ which is the ``weight" of $\V i$ in $GPM_{n,j}$. Define $\dega{i}{n}:=\degb{i}{n}{m}$ and $\wdb{i}{n}:=\wda{i}{n}{m}$. 

Also define
\begin{equation*}
    L(n):=\sum\limits_{1\leq i\leq n-1, D(\V i,\V n)\leq r}\wdb{i}{n-1}+m(2+\delta)
\end{equation*}
as the sum of weights in the detection area of $\V n$, and 
\begin{equation*}
    L_i(n)=\sum\limits_{1\leq j\leq n, D(\V i,\V j)\leq r}\wdb{j}{n}
\end{equation*}
as the sum of weights in the detection area of $\V i$ at time $n$. When $i>n$, we assume that we know the position of $\V i$ in advance when calculating $L_i(n)$. 

Let $v_{n,i}$ be the vertex to which the $i$-th edge from $\V n$ connects, and let $e_{n,i}=(\V n,v_{n,i})$ be that edge. Let $\mathcal{F}_{n,i}$ be the sigma field generated by the first $n$ vertices and edges up to $e_{n,i}$, and define $\mathcal{F}_n:=\mathcal{F}_{n,m}$.
Let $\mathcal{G}_{i,n}$ be the sigma field generated by $\mathcal{F}_i$ and the position of vertex $\V n$.

For any two positive sequences $\{a_n\}$ and $\{b_n\}$, we write equivalently $a_n=O(b_n)$, $b_n=\Omega(a_n)$, $a_n\lesssim b_n$ and $b_n\gtrsim a_n$ if there exists a positive absolute constant $c$ such that $a_n/b_n\leq c$ holds for all $n$. We write $a_n=o(b_n)$, $b_n=\omega(a_n)$, $a_n\ll b_n$, and $b_n\gg a_n$ if $a_n/b_n\to 0$ as $n\to\infty$. We write $a_n=\Theta(b_n)$, if both $a_n=O(b_n)$ and $a_n=\Omega(b_n)$. We extend these notations to two sequences of positive random variables if they satisfy the condition almost surely. For example, for two sequences of random variables $\{X_n\}$ and $\{Y_n\}$, we write $X_n=o(Y_n)$ if $X_n/Y_n\to 0$ as $n\to\infty$ almost surely.

Unless otherwise specified, all parameters $c$, $C$, $C_i$ in this paper are constants that depend only on $m$ and $\delta$, and their values may change from line to line.  

\subsection{Main results}
In this section, we present the main results of this paper. Our first result characterizes the asymptotic number of triangles in the GPM.
\begin{theorem}\label{main-triangle-count}
Suppose $p$ is a constant. Let $T_n$ be the number of triangles in $GPM_n$, i.e.
\begin{align*}
    T_n=\#\{1\leq a<b<c\leq n,\ 1\leq t_1,t_2,t_3\leq m \mid v_{b,t_1}=V_a,\ v_{c,t_2}=V_a,\ v_{c,t_3}=V_b\}.
\end{align*}
For $\V i,\V j\in S$, let $X$ be uniformly chosen on $S$, define 
\begin{align*}
    &f(\V i,\V j):=p^{-1}\p(\V i,\V j\in B(X,r)),\\
    &F_p:=\e(f(\V i,\V j)\mid \V j\in B(\V i,r)).
\end{align*}
Then, when $n\rightarrow \infty$,  
\begin{align*}
    T_n= (1+o(1))\frac{m(m-1)(1+\delta)^2(m(1+\delta)+1)F_p \log(n)}{(2+\delta)\delta^2p}.
\end{align*}
\end{theorem}
This quantifies how the geometric constraints amplify clustering. Compared to the PAM case (proven in \cite[Proposition 4.3]{EGGEMANN2011953}), the number of triangles in GPM is asymptotically equal to the number of triangles in PAM multiplied by $p^{-1}F_p$. Note that $F_p$ only depend on the parameter $p$ and the dimension $d$, and converges to a nonzero constant as $p\rightarrow 0$ for fixed $d$. Thus, the spatial structure increases the number of triangles by a factor of $\Theta(p^{-1})$.

An alternative interpretation of this result is as follows: for three vertices in the GPM to form a triangle, the distance between each pair must be less than $r$, and the probability of this occurring is $p^2 F_p$. For each triple of vertices satisfying this condition, the probability that they form a triangle is higher by a factor of $p^{-3}$ than in the PAM.

Our next result shows the growth rate and convergence of the maximum degree in the GPM.
\begin{theorem}\label{main-maximum-degree}
    (1)When the dimension $d$ is finite and fixed, there exist a constant $M$ and a function $f:\mathbb{R}^3\rightarrow [0,1]$ such that for any $n>Mp^{-1}\log(p^{-1})$ we have:
    \begin{align*}
        \p\left(C_1\log(p^{-1})^{\frac{1+\delta}{2+\delta}}(np)^{\frac{1}{2+\delta}}\leq \max\limits_{1\leq i\leq n}(\dega{i}{n})\leq C_2\log(p^{-1})^{\frac{1+\delta}{2+\delta}}(np)^{\frac{1}{2+\delta}}\right)\geq f(C_1,C_2,p)
    \end{align*}
    and $f(C_1,C_2,p)\rightarrow 1$ when either of the following is satisfied:
    \begin{align}
        &(a) C_1\rightarrow 0,\ C_2\rightarrow \infty,\ p<1,\ or\\
        &(b) C_1< D_1,\ C_2> D_2,\ p\rightarrow 0. \label{eq:thm1.2-b}
    \end{align}
    Here $D_1,D_2>0$ are two constants that only depend on the dimension $d$ and the parameters $m,\delta$.
    
    (2)There exists a random variable $X$ such that 
    \begin{align*}
        \frac{\max\limits_{1\leq i\leq n}(\dega{i}{n})}{\log(p^{-1})^{\frac{1+\delta}{2+\delta}}(np)^{\frac{1}{2+\delta}}}\rightarrow X \quad a.s.
    \end{align*}
    when $n\rightarrow \infty$.
\end{theorem}
As shown in Theorem \ref{main-maximum-degree}, the maximum degree of the GPM grows sublinearly as $(np)^{1/(2+\delta)}$, similar to the PAM, but modulated by a logarithmic factor dependent on $\delta$. Intuitively speaking, the additional logarithmic factor arises due to the following reason: the degree of each vertex in the GPM grows locally like a vertex in the PAM with $np$ vertex; however, the area of the detection region of each vertex is $p$, so we have $O(p^{-1})$ vertices that grow nearly independently, and the logarithmic factor arises since we need to take the maximum over the degree of these vertices.

The next theorem concerns the threshold of the connectivity probability in the GPM.
\begin{theorem}\label{main-connectivity}
     There exist functions $f_1,f_2:\mathbb{R}^{+}\rightarrow [0,1)$ such that when $n\geq 2$, then 
     \begin{align*}
         f_1(p^{\frac{m}{m-1}}n)\leq \p(GPM_n \text{ is connected})\leq f_2(p^{\frac{m}{m-1}}n),
     \end{align*}
     and
     \begin{align*}
         \lim_{x\rightarrow 0}f_2(x)=0, \lim_{x\rightarrow \infty}f_1(x)=1.
     \end{align*}
 \end{theorem}
It turns out that the emergence of isolated vertices is the main factor that makes the GPM disconnected. When the network size $n$ reaches the order $p^{-m/(m-1)}$, isolated vertices start to disappear and the probability of being connected jumps from near $0$ to near $1$. The critical scaling depends on both the geometric factor ($p$) and the out-degree of each vertex ($m$).

Finally, we examine the diameter of GPM. Recall the definition of $p_n(V_i,V_j)$ and $diam(GPM_n)$ in Section \ref{section-notation}.
\begin{theorem}\label{main-diameter}
    Suppose $p$ is a constant. There exist constants $C_1,C_2$ such that for any $p>0$, almost surely, for $n$ sufficiently large, we have
    \begin{align*}
        C_1\log(n)\leq diam(GPM_n)\leq C_2\log(n).
    \end{align*}
\end{theorem}
The network exhibits a small-world phenomenon. Despite spatial constraints, the diameter grows logarithmically with the size of the network when the number of vertices is large enough, similar to the PAM with $\delta>0$, which was proved in \cite{Dommers2010}. However, the evolution of the diameter when the number of vertices is small is still unknown. From computer simulations, we observe that when vertices are added to the GPM, the diameter first increases rapidly to a local maximum, then decreases until reaching a steady state, and finally grows logarithmically.
    \subsection{Extension to general preference function}
         \cite{Flaxman01012007} extends the geometric preferential attachment model to a more general case. We can similarly revise our model by changing the edge connection probability in \eqref{edgeconnectprob} to
        \begin{align*}
            \p(e_{n,i+1}=(\V n,\V k)\mid GPM_{n,i})=
    \begin{cases}
        \frac{f(D(\V n,\V k))(\degb{k}{n}{i}+m\delta)}{\sum\limits_{j=1}^{n-1}f(D(\V n,\V j))(\degb{j}{n}{i}+m\delta)+m\delta+m-i},\quad & \text{if}\quad 1\leq i \leq n-1 \\
        \frac{\degb{n}{n}{i}+m\delta+m-i}{\sum\limits_{j=1}^{n}f(D(\V n,\V j))(\degb{j}{n}{i}+m\delta)+m\delta+m-i},\quad & \text{if}\quad i = n
    \end{cases}
        \end{align*}
        where $f:\mathbb{R}^{\geq 0}\rightarrow \mathbb{R}^{\geq 0}$ is a given function with $f(0)=1$. Our model then becomes a special case of this generalized model with $f(x)=\mathbf{1}_{{x\leq r}}$. Most of the methods in this paper remain valid under this generalization. Suppose $z$ is any fixed vertex on $S$ and 
        \begin{align*}
            &p=\int_{S}f(D(u,z))du,\\
            &F=\int_{S^2}f(D(u,v))f(D(v,z))f(D(u,z))dudv.
        \end{align*} 
        It is easy to verify that the values of $p$ and $F$ do not depend on the choice of the position of the vertex $z$. Then we can give an estimate for the number of triangles:
        \begin{align*}
            T_n= (1+o(1))\frac{F}{p^3}\frac{m(m-1)(1+\delta)^2(m(1+\delta)+1)\log(n)}{(2+\delta)\delta^2}.
        \end{align*}
And the maximum degree in this network grows polynomially with the exponent $\frac{1}{2+\delta}$, and the diameter grows logarithmically, similar to the PAM. However, because the edge connection probability for a new vertex depends on all previous vertices, it is more challenging to analyze the early evolution of the process, especially when $\int_{S}f(D(u,z))^2du=\Omega(p^2)$. In this case, the concentration of $L(n)$ becomes weaker, and we cannot provide an estimate for the maximum degree that is accurate up to a constant.

For the connectivity threshold, note that the probability that $\V n$ is an isolated vertex in $GPM_n$ is $O((np)^{-m})$, and the argument in Section \ref{disconnectss} remains valid. However, the situation in the connected regime may differ. For example, when $f(x)=a \cdot \mathbf{1}_{{x\leq r}}$ satisfies $p=\int{S}f(D(u,z))du=1$, then for each vertex, the probability that its distance to the new vertex is less than $r$ is only $a^{-1}$. Therefore, with high probability, the network will not be connected until $n=\Omega(a\log a)$. Estimating the connection time and establishing a result analogous to equation (1) in Theorem $\ref{main-maximum-degree}$ remains an open problem.

\section{Preliminary estimates on candidate set weights}\label{sectionln}

The main difficulty in analyzing the GPM is that the total weight in the candidate set (i.e., $L(n)$) is a random variable compared to a constant in the PAM. In this section, we address this problem by proving two theorems about the concentration of $L(n)$, the latter of which is a strengthening of the former. These theorems are similar to Lemma 4.3 in \cite{Flaxman01012006} but more suitable for smaller $n$, and provide a probabilistic error bound for any arbitrary error tolerance $\varepsilon>0$.
\begin{theorem}\label{thm2.1}
For any $\varepsilon>0$, there exist constants $C_1,C_2,C_3>0$ such that for any $n,p>0$, we have
\begin{align*}
    \p((2+\delta-\varepsilon)mpn\leq L(n)\leq (2+\delta+\varepsilon)mpn)\geq 1-C_1\exp(-C_2(np)^{C_3}). 
\end{align*}
When the dimension $d$ is finite and fixed, then $C_3=1$. 
\end{theorem}

This is a theorem about estimating the large deviation of $L(n)$. Notice that for each previous vertex $\V i$,
\begin{align*}
    \p(D(\V i,\V n)\leq r)=p.
\end{align*}
Then we have
\begin{align*}
    \e(L(n))=m(2+\delta)+\sum\limits_{i=1}^{n-1}p\wdb{i}{n-1}=(2+\delta)mpn.
\end{align*}

The main idea of the proof is similar to the mathematical recursion. We will first prove a rough bound as in Lemma \ref{lem2.3} and Lemma \ref{lem2.8}. Then we will show that for a new vertex, if all vertices close to the new vertex satisfy that bound, then we can give a more precise bound for the new vertex. Repeating the process gives the desired bound.

For a vertex $\V i$ and $j\geq 1$, define $S_{i,n}^j:=\{ x\in S\mid d_n(\V i,x)\leq j\}$ and $S_{i,n}^0:=\{\V i\}$. We use $G_n(S_{i,n}^j)$ to represent the number of vertices in $S_{i,n}^j\cap \gV{n}$, recall that $A(X)$ is the area of $S$. Then it is easy to see that
    \begin{align}\label{rggrecursion}
        A(S_{i,n}^{j+1})\leq pG_n(S_{i,n}^{j}).
    \end{align}
If we assume that we know the position of the vertices in $S_{i,n}^{j}$, we can give a bound for $A(S_{i,n}^{j+1})$ by \eqref{rggrecursion}, and the position of the vertices in $S\setminus S_{i,n}^{j}$ follows a uniform distribution. Then we can use recursion to estimate $G_n(S_{i,n}^j)$, which is shown in the following lemma.
\begin{lemma}\label{sijexpbound}
For $j\geq 1$,
      \begin{align*}
        \e(\exp((2pn+2)^{-(j-1)} G_n(S_{i,n}^{j})))\leq  \exp(2pn+1).
    \end{align*}
\end{lemma}
\begin{proof}
    Notice that condition on the vertices in $S_{i,n}^j$,  the distribution of $G_n(S_{i,n}^{j+1})-G_n(S_{i,n}^{j})$ is a binomial distribution $B\left(n-G_n(S_{i,n}^j),\frac{A(S_{i,n}^{j+1})-A(S_{i,n}^j)}{1-A(S_{i,n}^j)}\right)$, Which can be stochastic dominated by $B(n,pG_n(S_{i,n}^j))$. Then we have 
    \begin{align*}
        \e(\exp(\theta G_n(S_{i,n}^{j+1})))\leq  \e(\exp(\theta G_n(S_{i,n}^{j})+(\text{e}^{\theta}-1)pnG(S_{i,n}^{j}))).
    \end{align*}
    Take $\theta_j=(2pn+2)^{-(j-1)}$, then for all $k\geq 1$ we have
    \begin{align*}
        \e(\exp(\theta_{j+1} G_n(S_{i,n}^{j+1})))\leq  \e(\exp(\theta_{j} G_n(S_{i,n}^{j}))).
    \end{align*}
    Then 
    \begin{align*}
        \e(\exp(\theta_{j} G_n(S_{i,n}^{j})))\leq \e(\exp(G_n(S_{i,n}^{1})))\leq \exp(2pn+1).
    \end{align*}
\end{proof}
The following two lemmas provide rough lower and upper bounds.
\begin{lemma}\label{lem2.3}
    (Rough lower bound) For any constant $\varepsilon>0$,  
    \begin{align*}
        \p(L(n)\geq (1+\delta-\varepsilon)mp(n-1)+m(2+\delta))\geq 1- \exp\left(-\frac{\varepsilon^2}{2(1+\delta)^2}pn\right)
    \end{align*}
    for any $p,n>0$.
\end{lemma}

\begin{proof}
    Since $L(n)\geq (1+\delta)mG_n(S_{n,n}^1)+m$, and the distribution of $S_{n,n}^1-1$ is the binomial distribution $B(n-1,p)$. Then $\e(\exp(\theta G_n(S_{n,n}^1)))= \exp(\theta+(\text{e}^{\theta}-1)p(n-1))$. Taking $\theta=\frac{-\varepsilon}{1+\delta}$, we get
    \begin{align*}
         \p(L(n)\leq (1+\delta-\varepsilon)mp(n-1)+m(2+\delta))
         \leq &\p\left(S_{n,n}^1\leq \frac{1+\delta-\varepsilon}{1+\delta}p(n-1)+1\right)\\
         \leq& \exp\left(-\frac{\varepsilon^2}{2(1+\delta)^2}p(n-1)\right).\qedhere
    \end{align*}
\end{proof}

\begin{lemma}\label{lem2.4}
    (Rough upper bound) There exist constants $C_1,C_2>0$ such that 
    \begin{align*}
        \p(L(n)\leq 20(2+\delta)m(pn)^2)\geq 1-20\exp(-pn)
    \end{align*}
    for any $p,n>0$.
\end{lemma}

\begin{proof}
    Since each edge intersecting $S_{n,n}^1$ must come from a vertex in $S_{n,n}^2$, so 
    \begin{equation*}
        L(n)\leq mG_n(S_{n,n}^2)+(1+\delta)mG_n(S_{n,n}^1).
    \end{equation*}
    Then, by Lemma \ref{sijexpbound},
    \begin{align*}
         \e(\exp((2pn+2)^{-1} G_n(S_{i,n}^2)))\leq  \exp(2pn+1).
    \end{align*}
    By Markov's inequality,
    \begin{align*}
        &\p(L(n)\geq 20(2+\delta)m(pn)^2)\leq \p(G_n(S_{i,n}^2)\geq 20(pn)^2)\\
        \leq &\min\left(1,\exp\left(-\frac{10(pn)^2}{pn+1}+2pn\right)\right)\leq 20\exp(-pn)
    \end{align*}
    which completes the proof of the lemma.
\end{proof}
The purpose of the next two lemmas is to help us to bound the large deviations of $L(n)$. 
\begin{lemma}\label{me}
    For constants $0<C_1<1, C_2,C_3\geq 0$, let $X_n$ be any random process satisfying 
    \begin{align}
        &\e(X_{n+1}\mid \mathcal{F}_{X_n})\leq (1+C_1n^{-1})X_n+C_2,\label{mec1}\\
        &X_{n+1}=X_n \text{ or }X_n+1\leq X_{n+1}\leq X_n+C_3.\label{mec2}
    \end{align}
    Where $\mathcal{F}_{X_n}$ is the sigma field generated by $\{X_i,i\leq n\}$. Then for any $n,M,N>0$, $C_1\neq \frac{1}{2}$,
    \begin{align*}
        &\p\left(\left.\exists \ N\leq i\leq n\text{~such~that~}  X_i\geq \frac{C_2}{1-C_1}i+M\prod\limits_{k=N}^{i-1}(1+C_1k^{-1}) \right|\mathcal{F}_{X_N}\right)\\
        \leq &\exp\left(-\frac{\max(0,M-X_N+\frac{C_2}{1-C_1}N)^2}{32C_3^2\max(M,\frac{C_2}{(1-C_1)|1-2C_1|}N^{\min(1,2C_1)}n^{1-\min(1,2C_1)})}\right).
    \end{align*}
\end{lemma}
\begin{proof}
    Define $c_n=\prod\limits_{i=N}^{n}(1+C_1i^{-1})=(1+O(N^{-1}))\left(\frac{n}{N}\right)^{C_1}$, $Y_n=c_{n-1}^{-1}(X_n-\frac{C_2}{1-C_1}n)$, then
    \begin{align*}
        \e(Y_{n+1}\mid \mathcal{F}_{X_n})&\leq c_n^{-1}\left((1+C_1n^{-1})X_n+C_2-\frac{C_2}{1-C_1}(n+1)\right)\\
        &=c_{n-1}^{-1}\left(X_n-\frac{C_2}{1-C_1}n\right)=Y_n.
    \end{align*}
    So $Y_n$ is a supermartingale. Now we want to make an exponential bound. Suppose $\tau_{M}=\inf\{n\geq N\mid Y_n\geq M\}$, then 
    \begin{align*}
        X_n\geq \frac{C_2}{1-C_1}n+M\prod\limits_{i=N}^{n-1}(1+C_1i^{-1}) \text{ implies } Y_n\geq M.
    \end{align*}
    For any $N\leq n<\tau_M$ we have
    \begin{align*}
        X_n\leq c_nM+\frac{C_2}{1-C_1}n.
    \end{align*}
    By \eqref{mec2}, we have
    \begin{align*}
        \p(X_{n+1}>X_n)\leq C_2+C_1n^{-1}X_n\leq C_2+C_1c_nMn^{-1}+\frac{C_1C_2}{1-C_1}.
    \end{align*}
    Notice that $X_n\leq X_{n+1}\leq X_n+C_3$. Thus, conditioned on $\mathcal{F}_{X_n}$, $Y_{n+1}$ takes values in a set of diameters at most $C_3c_n^{-1}$. Taking $\theta\leq C_3^{-1}$, by $\text{e}^x\leq 1+x+x^2$ on $|x|\leq 1$, we have
    \begin{align*}
        \e(\exp(\theta(Y_{n+1}-Y_n))\mid \mathcal{F}_{X_n})\leq &1+2\theta^{2}c_n^{-2}\e((X_{n+1}-X_{n})^2\mid \mathcal{F}_{X_n}))\\
        \leq &\exp\left(2\theta^{2}c_n^{-2}C_3^2\left(C_1c_nMn^{-1}+\frac{C_2}{1-C_1}\right)\right),
    \end{align*}
    which gives
    \begin{align*}
        \e(\exp(\theta Y_n)\mid \mathcal{F}_{X_N})\leq \exp\left(\theta Y_N+\sum\limits_{i=N}^{n-1}2\theta^{2}c_i^{-2}C_3^2\left(C_1c_iMi^{-1}+\frac{C_2}{1-C_1}\right)\right).
    \end{align*}
     By Markov's inequality, when $C_1>\frac{1}{2}$, 
    \begin{align}\label{C_1large}
        \p(\tau_M\leq n)\leq &\exp\left(-\theta(M-Y_N)+4\theta^2C_3^2\left(M+\frac{C_2}{(1-C_1)(2C_1-1)}N\right)\right).
    \end{align}
     Taking $\theta=\frac{M-Y_N}{8C_3^2(M+\frac{C_2}{(1-C_1)(2C_1-1)}N)}$ in \eqref{C_1large} completes the proof of the lemma for $C_1>\frac{1}{2}$.    
    Similarly, when $C_1<\frac{1}{2}$,
    \begin{align}\label{C_1small}
        \p(\tau_M\leq n)\leq &\exp\left(-\theta(M-Y_N)+4\theta^2C_3^2\left(M+\frac{C_2}{(1-C_1)(1-2C_1)}N^{2C_1}n^{1-2C_1}\right)\right).
    \end{align}
    Taking $\theta=\frac{M-Y_N}{8C_3^2\left(M+\frac{C_2}{(1-C_1)(1-2C_1)}N^{2C_1}n^{1-2C_1}\right)}$ in \eqref{C_1small} completes the proof of the lemma for $C_1<\frac{1}{2}$.
\end{proof}
\begin{lemma}\label{me2}
    For constants $0<C_1<1, C_2,C_3\geq 0$, let $X_n$ be any random process satisfying 
    \begin{align}
        &\e(X_{n+1}\mid \mathcal{F}_{X_n})\geq (1+C_1n^{-1})X_n+C_2,\label{me2c1}\\
        &X_{n+1}=X_n \text{ or } X_n+1\leq X_{n+1}\leq X_n+C_3.\label{me2c2}
    \end{align}
    Where $\mathcal{F}_{X_n}$ is the sigma field generated by $\{X_i,i\leq n\}$. Then for any $n,M,N>0$, $C_1\neq \frac{1}{2}$,
    \begin{align*}
        &\p\left(\left.\exists \ N\leq i\leq n\text{~such~that~}  X_i\leq \frac{C_2}{1-C_1}i-M\prod\limits_{k=N}^{i-1}(1+C_1k^{-1}) \right|\mathcal{F}_{X_N}\right)\\
        \leq &2\exp\left(-\frac{\left(\max\left(0,M+\min\left(0,X_N-\frac{C_2}{1-C_1}N\right)\right)^2\right)}{32C_3^2\max\left(M,\frac{C_2}{(1-C_1)|1-2C_1|}N^{\min(1,2C_1)}n^{1-\min(1,2C_1)}\right)}\right).
    \end{align*}
\end{lemma}
\begin{proof}
    We only need to prove Lemma \ref{me2} for the case where $X_n$ satisfies
    \begin{align}
        &\e(X_{n+1}\mid \mathcal{F}_{X_n})= (1+C_1n^{-1})X_n+C_2, \ X_N\leq \frac{C_2}{1-C_1}N. \label{repeatxn}
    \end{align}
    Or we can construct a random process $X_n'$ based on $X_n$ such that $X_N'=\min(X_N,\frac{C_2}{1-C_1}N)$ and 
    \begin{align*}
        &\p(X_{n+1}'-X_n'=X_{n+1}-X_n|\mathcal{F}_{X_{n+1}})=\frac{C_2+C_1n^{-1}X_n'}{\e(X_{n+1}-X_n\mid \mathcal{F}_{X_n})},\\
        &\p(X_{n+1}'-X_n'=0|\mathcal{F}_{X_{n+1}})=1-\frac{C_2+C_1n^{-1}X_n'}{\e(X_{n+1}-X_n\mid \mathcal{F}_{X_n})}.
    \end{align*}
     Then $X_n'$ satisfies \eqref{repeatxn} and $X_n'\leq X_n$ for all $n\geq N$. So we only need to prove Lemma \ref{me2} for $X_n'$.
    
    The rest of the proof is similar to that of Lemma \ref{me}. Define $c_n=\prod\limits_{i=N}^{n}(1+C_1i^{-1})=(1+O(N^{-1}))(\frac{n}{N})^{C_1}$, $Y_n=c_{n-1}^{-1}(\frac{C_2}{1-C_1}n-X_n)$, then
    \begin{align*}
        \e(Y_{n+1}\mid \mathcal{F}_{X_n})&= c_n^{-1}\left(-(1+C_1n^{-1})X_n-C_2+\frac{C_2}{1-C_1}(n+1)\right)\\
        &=c_{n-1}^{-1}\left(\frac{C_2}{1-C_1}n-X_n\right)=Y_n,
    \end{align*}
    so $Y_n$ is a martingale. Now we want to make an exponential bound. Suppose $\tau_{M}=\inf\{n\geq N\mid |Y_n|\geq M\}$. By Lemma \ref{me},
    \begin{align*}
    \p(\exists \ N\leq i\leq n, \ Y_i\leq -M)\leq \exp\left(-\frac{\left(\min\left(M,M+X_N-\frac{C_2}{1-C_1}N\right)^2\right)}{32C_3^2\max\left(M,\frac{C_2}{(1-C_1)|1-2C_1|}N^{\min(1,2C_1)}n^{1-\min(1,2C_1)}\right)}\right).
    \end{align*}
    By \eqref{me2c2} we have 
    \begin{align*}
        \p(X_{n+1}>X_n)\leq C_2+C_1n^{-1}X_n= -C_1c_nY_nn^{-1}+\frac{C_2}{1-C_1}.
    \end{align*}
    Notice that $X_n\leq X_{n+1}\leq X_n+C_3$. Thus, conditioned on $\mathcal{F}_{X_n}$, $Y_{n+1}$ takes values in a set of diameters at most $C_3c_n^{-1}$. Taking $\theta\leq C_3^{-1}$, by $\text{e}^x\leq 1+x+x^2$ on $|x|\leq 1$, we have
    \begin{align*}
        \e(\exp(\theta(Y_{n+1}-Y_n))\mid \mathcal{F}_{X_n})\leq &1+2\theta^{2}c_n^{-2}\e((X_{n+1}-X_{n})^2\mid \mathcal{F}_{X_n})\\
        \leq &1+2\theta^{2}c_n^{-2}C_3^2\left(C_1c_nMn^{-1}+\frac{C_2}{1-C_1}\right).
    \end{align*}
    By the same argument as in Lemma \ref{me},
    \begin{align*}
        \p(\tau_M\leq n, Y_{\tau_M}\geq M)\leq \exp\left(-\frac{\left(\min\left(M,M+X_N-\frac{C_2}{1-C_1}N\right)^2\right)}{32C_3^2\max\left(M,\frac{C_2}{(1-C_1)|1-2C_1|}N^{\min(1,2C_1)}n^{1-\min(1,2C_1)}\right)}\right).
    \end{align*}
    which completes the proof of the lemma.
\end{proof}

In this section, we apply Lemmas \ref{me} and \ref{me2} to establish the concentration of $L(n)$ by setting $X_i = L_n(i)$. Note that when a new vertex is added, $L_n(i)$ increases by at most $m(2+\delta)$. In the following lemma, we choose the appropriate constants $C_1$ and $C_2$ so that $X_i$ satisfies the conditions of Lemmas \ref{me} and \ref{me2}, under the assumption that all vertices near $V_n$ obey the concentration bounds of Lemmas \ref{lem2.3} and \ref{lem2.4}. Specific details will be discussed later in this section.

We only care about the vertices near the new vertex; however, we cannot know in advance which vertices these are before the new vertex is placed. Therefore, we use the following lemma to bound the probability that all vertices near the new vertex satisfy the previous concentration bounds.
\begin{lemma}\label{lem2.7}
Let $A_i\in \mathcal{F}_{n-1},1\leq i\leq n-1$ be events satisfying $\max\limits_{1\leq i\leq n-1}(\p(A_i))\leq c$ for a constant $c\in (0,1)$, then for any integer $k\geq 1$,
\begin{align*}
    \p(\cup_{d_n(\V i,\V n)\leq k}A_i)\leq (2pn+2)^{k}c(-\log(c)+1).
\end{align*}
\end{lemma}
\begin{proof}
By Lemma \ref{sijexpbound}, $\p(G(S_{i,n}^k)\geq t(2pn+2)^k)\leq \text{e}^{-(t-1)(2pn+2)}$. And 
\begin{align*}
    \e(\mathbf{1}_{A_i}\mathbf{1}_{\{d_n(i,n)\leq k\}})
    &\leq \e(\mathbf{1}_{A_i}A(S_{i,n}^k))\\
    &\leq \e(\mathbf{1}_{A_i}pG(S_{i,n}^{k-1}))\\
    &\leq \int_0^{\infty}p\min(\p(A_i),\exp(\frac{-x}{(2pn+2)^{k-2}}+2pn+2))dx\\
    &\leq (2pn+1)^{k-1}pc+(2pn+1)^{k-2}pc(-\log(c)+1).
\end{align*}
Finally, with a union bound, we get 
    \begin{align*}
        \p(\cup_{d_n(i,n)\leq k}A_i)\leq \sum\limits_{1\leq i\leq n} \e(\mathbf{1}_{A_i}\mathbf{1}_{\{d_n(i,n)\leq k\}}) \leq (2pn+2)^{k}c(-\log(c)+1).\qedhere
    \end{align*}
\end{proof}
\begin{lemma}\label{lem2.8}
    There exist constants $C_1,C_2>0, 0<C_3\leq 1$ such that for any $p,n>0$, $m\in \mathbb{N}$,
    \begin{align}
        \p(L(n)\leq 2C_1mpn)\geq 1-2\exp(-C_2(pn)^{C_3}).\label{eqlem2.8}
    \end{align}
     In particular, when the dimension $d$ is finite and fixed, then \eqref{eqlem2.8} holds for $C_3=1$.
\end{lemma}

\begin{proof}
    Take $N=\lceil p^{-1}(pn)^{\frac{\delta}{3+6\delta}}\rceil$, $X_i=L_n(i)$. Define the bad event 
    \begin{align*}
        A&:=\{ L_n(N-1)\leq 20(2+\delta)m(pN)^2\}\\
        B_j&:=\{\forall \ N\leq i\leq j,\ d_i(\V i,\V n)\leq 2,\ L(i)\geq (1+\delta /2)mp(i-1)+m(2+\delta)\}.
    \end{align*}
    Recall that $\mathcal{G}_{i,n}$ is the sigma field generated by $\mathcal{F}_i$ and the position of $\V n$, then $B_j\in \mathcal{G}_{j,n}$. Note that conditioned on $\mathcal{F}_{N-1}$, $L_n(N-1)$ has the same distribution as $L(N)-m(2+\delta)$. Then by Lemmas \ref{lem2.3}, \ref{lem2.4} and \ref{lem2.7},
    \begin{align*}
        &\p(B_{n-1}^c)\leq (2pn+2)^3\exp\left(-\frac{\delta^2}{8(1+\delta)^2}pN\right),\\
        &\p(A^c)\leq 20\exp(-pN).
    \end{align*}
    Define $\tau=\inf\{s\geq N|B_s^c\}$, $\tau_i=\min(\tau,i)$ and $X_i:=L_n(\tau_{i}-1)$, then $X_i\in \mathcal{G}_{i-1,n}$, when $i\leq n-2$, for a new vertex $\V{i+1}$ we have
    \begin{align*}
        \p(\V{i+1}\in S_{n,i}^1\mid \mathcal{G}_{i,n})&=p,\\
        \p(v_{i+1,s}\in S_{n,i}^{1},B_{i+1}\mid \mathcal{G}_{i,n})&=\sum\limits_{\substack{1\leq j\leq i\\\V j\in S_n^1}}\e\left(\frac{\wda{j}{i+1}{s-1}}{L(i+1)-m+s-1}\mathbf{1}_{\{D(\V j,\V {i+1})\leq r\}}\mathbf{1}_{B_{i+1}}\mid \mathcal{G}_{i,n}\right)\\
        &=\sum\limits_{\substack{1\leq j\leq i\\\V j\in S_n^1}}\e\left(\frac{\wdb{j}{i}}{L(i+1)-m}\mathbf{1}_{\{D(\V j,\V {i+1})\leq r\}}\mathbf{1}_{B_{i+1}}\mid \mathcal{G}_{i,n}\right)\\
        &\leq \sum\limits_{\substack{1\leq j\leq i\\\V j\in S_n^1}}\frac{\wdb{j}{i}}{(1+\delta /2)m(i+1)}=\frac{L_n(i)}{(1+\delta /2)m(i+1)}.
    \end{align*}        
    Therefore, we can give a recursion for $X_i$ in the form of Lemma \ref{me}: 
    \begin{align*}
        &\e(X_{i+1}\mid \mathcal{G}_{i-1,n})\leq \left(1+\frac{1}{(1+\delta /2)i}\right)X_i+(1+\delta)mp,\\
        &X_{i+1}=X_i \text{ or }X_i+1\leq X_{i+1}\leq X_i+m(2+\delta).
    \end{align*}
    Take $M=\left(C_1-\frac{(1+\delta)(2+\delta)}{\delta}\right)mpn^{\frac{\delta}{2+\delta}}N^{\frac{2}{2+\delta}}$ in Lemma \ref{me}, 
    \begin{align*}
        \p(L(n)\geq 2C_1mpn)&\leq \p(B_{n-1}^c)+\p(X_n\geq 2C_1mpn-m(2+\delta))\\
        &\leq \p(B_{n-1}^c)+\p(A)+\exp\left(-\frac{(M-20(2+\delta)(pN)^2)^2}{32m^2(2+\delta)^2M}\right).
    \end{align*}
    For $C_1$ and $pn$ sufficiently large, the right side can be dominated by $2\exp(-C_2(pn)^{C_3})$.
        
    When the dimension $d$ is finite and fixed, then $L(n)\leq (2+\delta)m(1+G_n(S_{n,n}^2))$ and $G_n(S_{n,n}^k)$ (note that $S_{n,n}^k \subset B(n,kr)$) can be stochastically dominated by a binomial distribution $B(n,k^{d} p)$. This immediately gives \eqref{eqlem2.8} for $C_3=1$. 
\end{proof}

\begin{lemma}\label{lem3.8}
    Suppose that there exist constants $C_1,C_2,C_3,C_4>0$ such that $1<C_1<2+\delta <C_2, C_4\leq 1$ and $\p(C_1mpn+m\leq L(n)\leq C_2mpn)\geq 1-2\exp(-C_3(pn)^{C_4})$ for any $p,n>0$. Then for any $\varepsilon>0$, there exists a constant $C_3'>0$ satisfying for any $p,n>0$,
    \begin{align*}
        &\p\left(\left((1+\delta)\frac{C_2}{C_2-1}-\varepsilon\right)mpn+m\leq L(n)\leq \left((1+\delta)\frac{C_1}{C_1-1}+\varepsilon\right)mpn\right)\\
        \geq &1-2\exp(-C_3'(pn)^{C_4}).
    \end{align*}  
\end{lemma}

\begin{proof}
Define $D_1=(1+\delta)\frac{C_2}{C_2-1}-\varepsilon$, $D_2=(1+\delta)\frac{C_1}{C_1-1}+\varepsilon$. Take $c$ to be a small constant, $N=\lceil cn\rceil$, 
\begin{align*}
    A&:=\{  C_1mpN-m(2+\delta)\leq L_n(N-1)\leq C_2mpN\}\\
    B_j&:=\{\forall \ N\leq i\leq j,\ C_1mpi+m \leq L(i)\leq C_2mpi\}.
\end{align*}
Define $\tau=\inf\{s\geq N|B_s^c\}$, $\tau_i=\min(\tau,i)$, and $X_i:=L_n(\tau_i-1)$. Then
\begin{align*}
    &\e(X_{i+1}\mid \mathcal{G}_{i-1,n})\leq \left(1+\frac{1}{C_1i}\right)X_i+(1+\delta)mp,\\
    &X_{i+1}=X_i\text{ or }X_i+1\leq X_{i+1}\leq X_{i}+(2+\delta)m.
\end{align*}
Take $M=\frac{1}{2} \varepsilon mpn c^{\frac{1}{C_1}}$ in lemma \ref{me}, then for $np$ sufficiently large, 
\begin{align*}
    \p(X_n\geq D_2mpn-m(2+\delta),A)\leq 2\exp\left(-\frac{\left(\frac{1}{2}\varepsilon mpn c^{\frac{1}{C_2}}-O(cpn)\right)^2}{O(M)}\right).
\end{align*}
 Similarly, define $X_i'=L_n(\tau_i-1)+m(2+\delta)(i-\tau_i)$, then
\begin{align*}
    &\e(X_{i+1}'\mid \mathcal{F}_i)\geq \left(1+\frac{1}{C_2i}\right)X_i'+(1+\delta)mp,\\
    &X_{i+1}'=X_i' \text{ or } X_i'+1\leq X_{i+1}'\leq X_{i}+(2+\delta)m.
\end{align*}
Taking $M=\frac{1}{2}\varepsilon mpn c^{\frac{1}{C_2}}$ in lemma \ref{me2}, for $pn$ sufficiently large, we can get
\begin{align*}
    \p(X_n'\leq D_1mpn+m,A)\leq 2\exp\left(-\frac{\left(\frac{1}{2}\varepsilon mpn c^{\frac{1}{C_2}}-O(cpn)\right)^2}{O(M)}\right).
\end{align*}
    Notice that $X_n'\leq L(n)\leq X_n+m(2+\delta)$ on $B_{n-1}$, then
\begin{align*}
    &1-\p\left(\left((1+\delta)\frac{C_2}{C_2-1}-\varepsilon\right)mpn+m\leq L(n)\leq \left((1+\delta)\frac{C_1}{C_1-1}+\varepsilon\right)mpn\right)\\
    \leq &\p(A^c)+\p(B_{n-1}^c)+\p(X_n\geq D_2mpn-m(2+\delta),A)+\p(X_n'\leq D_1mpn+m,A).
\end{align*}
The right side can be dominated by $2\exp(-C_3'(pn)^{C_4})$.
\end{proof}

Now we are ready to prove Theorem \ref{thm2.1}. 
\begin{proof}[Proof of Theorem \ref{thm2.1}]
By Lemmas \ref{lem2.3} and \ref{lem2.8}, there exist constants $C_1,C_2,C_3,C_4$ such that $1<C_1< C_2$, for any $p,n>0$,
\begin{align*}
    \p(C_1mpn+m\leq L(n)\leq C_2mpn)\geq 1-2\exp(-C_3(pn)^{C_4}).
\end{align*}
Then by Lemma \ref{lem3.8}, for any small $\varepsilon>0$, there exists a constant $C_3'$ such that
\begin{align*}
        &\p\left(\left((1+\delta)\frac{C_2}{C_2-1}-\varepsilon\right)mpn+m\leq L(n)\leq \left((1+\delta)\frac{C_1}{C_1-1}+\varepsilon\right)mpn\right)\\
        \geq& 1-2\exp(-C_3'(pn)^{C_4}).
    \end{align*}
    Notice that for any $\varepsilon'>0$, we can take $\varepsilon$ sufficiently small such that 
    \begin{align*}
        \frac{(1+\delta)^2C_2-\varepsilon(1+\delta)(C_2-1)}{1+\delta C_2-\varepsilon(C_2-1)}+\varepsilon<\frac{(1+\delta)^2C_2}{1+\delta C_2}+\varepsilon',\\
        \frac{(1+\delta)^2C_1+\varepsilon(1+\delta)(C_2-1)}{1+\delta C_1+\varepsilon(C_2-1)}-\varepsilon>\frac{(1+\delta)^2C_1}{1+\delta C_1}-\varepsilon'.
    \end{align*}
    Replace $C_1$ with $(1+\delta)\frac{C_2}{C_2-1}-\varepsilon$ and $C_2$ with $(1+\delta)\frac{C_1}{C_1-1}+\varepsilon$, using Lemma \ref{lem3.8} again, we get
    \begin{align*}
        &\p\left(\left(\frac{(1+\delta)^2C_1}{1+\delta C_1}-\varepsilon'\right)mpn+m\leq L(n)\leq \left(\frac{(1+\delta)^2C_2}{1+\delta C_2}+\varepsilon'\right)mpn\right)\\
        \geq& 1-2\exp(-C_3''(pn)^{C_4}).
    \end{align*}
    for some constant $C_3''$.
    Notice that $\frac{(1+\delta)^2x}{1+\delta x}<x$ if and only if $x>2+\delta$. As a result, by this process, we can use Lemma \ref{lem3.8} to improve any existing bound. Using Lemma \ref{lem3.8} several times, we can complete the proof of Theorem \ref{thm2.1}.
\end{proof}
By Theorem \ref{thm2.1}, we can bound the probability of $L(n)$ deviating from a fixed proportion of its expectation. However, in Section \ref{maxdegsec} we need to estimate the probability when the deviation $\varepsilon$ depends on $n$ and converges to $0$ when $n\rightarrow \infty$. Therefore, we need a strengthening of Theorem \ref{thm2.1} as follows. 
\begin{theorem}\label{thm2.10}
    Let $0<C_3 \leq 1$ be the constant in Lemma \ref{lem2.8}, then there exist constants $C_1,C_2,\alpha>0$ such that for any $0<\varepsilon<1$, we have
    \begin{align}\label{eqthm2.10}
        \p((2 +\delta-\varepsilon)mpn+m \leq L(n) \leq (2 + \delta +\varepsilon)mpn)\geq 1- C_1\exp(-C_2\varepsilon^{\alpha}(pn)^{C_3}).
    \end{align}
    In particular, when the dimension $d$ is finite and fixed, we have $C_3=1$.
\end{theorem} 
    A direct deduction of this theorem that will be used frequently in this paper is the following proposition.
    \begin{proposition}\label{prop:2.11}
        There exists a constant $c>0$ such that
    \begin{align*}
        \p((2 +\delta)mpn-(pn)^{1-c}+m \leq L(n) \leq (2 + \delta)mpn+(pn)^{1-c})\geq 1- C_1\exp(-C_2(pn)^{c}).
    \end{align*}
    \end{proposition}

The difference between Theorem \ref{thm2.1} and Theorem \ref{thm2.10} is that the parameters chosen do not depend on $\varepsilon$, and we give an estimate of the situation when $\varepsilon$ grows with $n$.

Note that by Theorem \ref{thm2.1}, when $\varepsilon_0$ is a constant, there exist constants $C_1,C_2,C_3>0$ such that 
    \begin{equation*}
        \p((2+\delta-\varepsilon_0)mpn \leq L(n) \leq (2+\delta +\varepsilon_0)mpn)\geq 1 - C_1\exp(-C_2(pn)^{C_3}).
    \end{equation*}
    And $C_3=1$ when the dimension $d$ is finite and fixed. Then we will use the following lemma to reduce $\varepsilon_0$ and complete the proof of Theorem \ref{thm2.10}.
    \begin{lemma}\label{lem2.12}
        Let $c=(\frac{\delta}{8+8\delta^2})^{2}$, $N=\lceil cn\rceil$. For $\varepsilon\leq \frac{\delta}{4}$, define the events 
        \begin{align*}
        A_{(s,n,\varepsilon)}&:=\{ L_s(n)\in [(2+\delta-\varepsilon)mpn+m,(2+\delta+\varepsilon)mpn-m(2+\delta)]\}, \\
        A_{(s,n,\varepsilon)}'&:=\{ L_s(n)\in [(2+\delta-\varepsilon)mpn-m(2+\delta),(2+\delta+\varepsilon)mpn+m(2+\delta)]\}, \\
        B_{(s,N,n,\varepsilon)}&:=\{\forall \ N \leq i \leq n,\ d_i(\V i,\V s)\leq2,  L(i)\in [(2+\delta-\varepsilon)mpi+m,(2+\delta+\varepsilon)mpi].
         \end{align*}
         Then there exists a constant $C=C(m,\delta)$ such that 
         \begin{equation*}
             \p\left(A_{(s,n,\frac{\varepsilon}{1+\delta/2})}^c\cap B_{(s,\lceil cn \rceil,n-1,\varepsilon)}\cap A_{(s,\lceil cn \rceil -1,\varepsilon)}'\right)\leq 4\exp(-C\varepsilon^2pn)
         \end{equation*}
         hold for any $s\geq n$.
    \end{lemma}
    \begin{proof}
        Define $\tau=\inf\{t\geq \lceil cn \rceil\mid B_{(s,\lceil cn \rceil,t,\varepsilon)}^c\}$, $\tau_i=\min(\tau,i)$, and $X_i=L_s(\tau_i-1)$. Then $X_i\in \mathcal{G}_{i-1,s}$, and we have
        \begin{align*}
            &\e(X_{i+1}\mid \mathcal{G}_{i-1,s})\leq \left(1+\frac{1}{(2+\delta-\varepsilon)i}\right)X_i+mp(1+\delta),\\
            &X_{i+1}=X_i \text{ or } X_i+1\leq X_{i+1}\leq X_i+m(2+\delta).
        \end{align*}
        Then Lemma \ref{me} gives 
        \begin{align*}
            &\p\left(X_n\geq \left(2+\delta+\frac{\varepsilon}{1+\delta/2}\right)mpn-2m(2+\delta),\ A_{(s,\lceil cn \rceil -1,\varepsilon)}'\right)\\
            \leq& \min\left(1,\exp\left(-\frac{(\varepsilon\frac{\delta}{4(1+\delta)^{2}} mpn c^{\frac{1}{2+\delta-\varepsilon}}-\varepsilon mpcn-O(1))^2}{O(pn)}\right)\right)\leq 2\exp(-C\varepsilon^2 pn).
        \end{align*}
        Similarly, define $X_i'=L_n(\tau_i-1)+m(2+\delta)(i-\tau_i)$, and we have
        \begin{align*}
            &\e(X_{i+1}'\mid \mathcal{F}_{i})\geq \left(1+\frac{1}{(2+\delta+\varepsilon)i}\right)X_i'+mp(1+\delta),\\
            &X_{i+1}'=X_i' \text{ or } X_i'+1\leq X_{i+1}'\leq X_i'+m(2+\delta).
        \end{align*}
        Lemma \ref{me2} gives
        \begin{align*}
            &\p\left(X_n'\leq \left(2+\delta-\frac{\varepsilon}{1+\delta/2}\right)mpn+m,\ A_{(s,\lceil cn \rceil -1,\varepsilon)}'\right)\\
            \leq &\min\left(1,\exp\left(-\frac{(\varepsilon\frac{\delta}{4(1+\delta)^{2}} mpn c^{\frac{1}{2+\delta-\varepsilon}}-\varepsilon mpcn-O(1))^2}{O(pn)}\right)\right)\leq 2\exp(-C\varepsilon^2pn).
        \end{align*}
        Since $X_n'\leq L_s(n)\leq X_n+m(2+\delta)$ on $B_{(s,\lceil cn \rceil,n-1,\varepsilon)}$, therefore
        \begin{align*}
               &\p\left(A_{(s,n,\frac{\varepsilon}{1+\delta/2})}^c\cap B_{(s,\lceil cn \rceil,n-1,\varepsilon)}\cap A_{(s,\lceil cn \rceil -1,\varepsilon)}'\right)\\
               \leq& \p\left(X_n\geq \left(2+\delta+\frac{\varepsilon}{1+\delta/2}\right)mpn-2m(2+\delta),\ A_{(s,\lceil cn \rceil -1,\varepsilon)}'\right)\\
               +&\p\left(X_n'\leq \left(2+\delta-\frac{\varepsilon}{1+\delta/2}\right)mpn+m,\ A_{(s,\lceil cn \rceil -1,\varepsilon)}'\right)\leq 4\exp(-C\varepsilon^2pn).
        \end{align*}
    \end{proof}
    Lemma \ref{lem2.12} shows that if all vertices $\V i$ near a new vertex $\V n$ have a well-concentrated for $L(i)$ with error $\varepsilon$, then we can derive a tighter concentration for $L(n)$ with error $\frac{1}{1+\delta/2} \varepsilon$. This allows us to prove the following lemma by mathematical induction.
    \begin{lemma}\label{goodlninduction}
        Let $0<C_3\leq 1$ be the constant in Lemma \ref{lem2.8}, then there exist constants $C_1,C_2,\alpha>0$ such that for any $0<\varepsilon<1$, we have
    \begin{align*}
        &\p((2 +\delta-\varepsilon)mpn+m \leq L(n) \leq (2 + \delta +\varepsilon)mpn)\\
        \geq& 1- (C_1pn+C_1)^{C_1(1-\log(\varepsilon))}\exp(-C_2\varepsilon^{\alpha}(pn)^{C_3}).
    \end{align*}
    In particular, when the dimension $d$ is finite and fixed, we obtain $C_3 = 1$.
    \end{lemma}
   First, we explain that this lemma is equivalent to Theorem \ref{thm2.10}. Since we only need to consider the situation of $\varepsilon^{2\alpha}(pn)^{C_3}\geq M$ for $M$ is a large constant, or the right side of \eqref{eqthm2.10} is less than $0$ by taking $C_1,\alpha$ in Theorem \ref{thm2.10} sufficiently large. Then
    \begin{align*}
        &C_1\log(C_1pn+C_1)(1-\log(\varepsilon))-C_2\varepsilon^{\alpha}(pn)^{C_3}\leq C'(pn)^{0.1C_3}\varepsilon^{-0.1\alpha}-C_2\varepsilon^{\alpha}(pn)^{C_3}\\
        \leq& C'M^{-0.6}(pn)^{C_3}\varepsilon^{\alpha}-C_2\varepsilon^{\alpha}(pn)^{C_3} \leq -C_2'\varepsilon^{\alpha}(pn)^{C_3}
    \end{align*}
    for some constants $C',C_2'>0$ that depend only on $M,C_1,C_2,\alpha$ chosen in Lemma \ref{goodlninduction}. As a result, to prove Theorem \ref{thm2.10} we only need to prove Lemma \ref{goodlninduction}.
    \begin{proof}[Proof of Lemma \ref{goodlninduction}]
        We will use mathematical induction. Let $c=(\frac{\delta}{8+8\delta^2})^{2}$, $\varepsilon_0=\frac{1}{2}$. By Theorem \ref{thm2.1}, for any constant $\varepsilon_0>0$, there exist constants $C_1> 10+\frac{10}{\log(1+\delta/2)}$, $0<C_2<\min(C,1)$ where $C$ is the constant in Lemma \ref{lem2.12} and $\alpha>2-\frac{\log(c)}{\log(1+\delta/2)}$ such that Lemma \ref{goodlninduction} holds for $\varepsilon_0\leq \varepsilon \leq 1$. Suppose Lemma \ref{goodlninduction} holds for $\varepsilon\geq \varepsilon_{k}$ for some $\varepsilon_k>0$, then for $\frac{\varepsilon_k}{1+\delta/2}\leq \varepsilon \leq \varepsilon_k$, by Lemma \ref{lem2.12},
        \begin{align*}
            &1-\p((2 +\delta-\varepsilon)mpn+m \leq L(n) \leq (2 + \delta +\varepsilon)mpn)\\
            \leq &1-\p\left(B_{(n,\lceil cn\rceil, n-1, (1+\delta/2)\varepsilon)}\right)+1-\p\left(A_{(n,\lceil cn\rceil-1,(1+\delta /2)\varepsilon)}'\right)+4\exp(-C\varepsilon^2np).
        \end{align*}
        Where $B_{(n,\lceil cn\rceil, n-1, (1+\delta/2)\varepsilon)}$ and $A_{(n,\lceil cn\rceil-1,(1+\delta /2)\varepsilon)}$ have the same definition as in Lemma \ref{lem2.12}. Let $N=\lceil cn \rceil $, $c_{n,\varepsilon}=(C_1pn)^{C_1(1-\log(\varepsilon))}\exp(-C_2\varepsilon^{\alpha}(pn)^{C_3})$, by the induction hypothesis and Lemma \ref{lem2.7},
        \begin{align*}
            \p\left(A_{(n,\lceil cn\rceil-1,(1+\delta /2)\varepsilon)}'\right)&\geq 1-c_{N,(1+\delta /2)\varepsilon},\\
            \p\left(B_{(n,\lceil cn\rceil, n-1, (1+\delta/2)\varepsilon)}\right)&\geq 1-(2pn+2)^{2}c_{N,(1+\delta /2)\varepsilon}(1+\log(c_{N,(1+\delta /2)\varepsilon}))\\
            &\geq 1-(2pn+2)^{3}c_{N,(1+\delta /2)\varepsilon}.
        \end{align*}
        Since $c_{n,\varepsilon}\geq (C_1pn+C_1)^{C_1\log(1+\delta/2)}c_{N,(1+\delta /2)\varepsilon}$, we have
        \begin{align*}
            c_{n,\varepsilon}\geq 1-\p(B_{(n,\lceil cn\rceil, n-1, (1+\delta/2)\varepsilon)})+1-\p(A_{(n,\lceil cn\rceil-1,(1+\delta /2)\varepsilon)}')+4\exp(-C\varepsilon^2np).
        \end{align*}
        which complete the proof of the induction.
    \end{proof}
    \begin{remark}
        We don't think the condition that $d$ is finite and fixed for $C_3=1$ in Theorems \ref{thm2.1} and \ref{thm2.10} is necessary. Actually, to remove this condition, we only need to prove $C_3=1$ in Lemma \ref{lem3.8} for any dimension. However, we cannot find a way to prove it, and it is left as an open problem.
    \end{remark}

\section{Preliminary estimates on subgraph probability}\label{sectionedgep}
To estimate the number of triangles, we want to know the probability of each specific subgraph appearing in GPM. Theorem \ref{thm3.1} is an estimate of it.  
Take a constant $c$ sufficiently small, for any parameters $n_1\leq n_2$, define the event 
\begin{equation}
    G_{n_1,n_2}= \{L(i)\in [(2+\delta)mpi-(pi)^{1-c},(2+\delta)mpi+(pi)^{1-c}],\ \forall \ n_1\leq i\leq n_2\}.\label{eq:edgeprob-lnevent}
\end{equation}
Note that by Proposition \ref{prop:2.11}, $\p(G_{n_1,n_2})=1-o(\exp(-(pn_1)^{c})$.
\begin{theorem}\label{thm3.1}
    Suppose $p$ is a constant. For any constant $k$, let ($\V{a_i},\ 1\leq i\leq k$), ($\V{b_i},\ 1\leq i\leq k$) be vertices such that $a_i\leq b_i\leq n, \ \forall \ 1\leq i\leq k$, let $1\leq t_1,t_2\dots, t_k\leq m$ such that $(b_i, t_i)$ are all distinct. Define $A:=\{D(\V{a_i},\V{b_i})\leq r,\ \forall \ 1\leq i\leq k\}$ and $a':=\inf\limits_{1\leq i\leq k}a_i$, then
    \begin{align}
        &\p(v_{b_i,t_i}=\V {a_i}\ \forall \ 1\leq i\leq k) \notag \\
        = &(1+O((a')^{-c}))\p(A)\prod\limits_{i=1}^{k}\left((1+\delta)m+\#\{j\mid 1\leq j<i,a_j=a_i\}\right)((2+\delta)m)^{-1}b_i^{-\frac{1+\delta}{2+\delta}}a_i^{-\frac{1}{2+\delta}}.\label{eq:3.1}
    \end{align}
\end{theorem}
\begin{proof}
    For any $a\leq b$, $1\leq s\leq m$, define the edge probability
    \begin{align*}
            \ell(a,b,s,t)=
            \begin{cases}
    \mathbf{1}_{\{v_{b,s}=\V a\}}& a,b\leq t \\
    \frac{\wdb{a}{t}}{(2+\delta)mpt^{\frac{1}{2+\delta}}b^{\frac{1+\delta}{2+\delta}}}& a\leq t<b \\
    \frac{1+\delta}{(2+\delta)p}b^{-\frac{1+\delta}{2+\delta}}(a-1)^{-\frac{1}{2+\delta}}&a,b>t
    \end{cases}
    \end{align*}
     First, we prove the upper bound. Define the random process 
    \begin{align}
        X_t=&\p(A\mid \mathcal{F}_t)\prod\limits_{i=1}^{k}\ell(a_i,b_i,t_i,t)(1+\frac{\mathbf{1}_{\{b_i>t\}}\#\{j\mid 1\leq j<i,a_j=a_i,b_j>t\}}{\wdb{a_i}{t}\mathbf{1}_{\{a_i\leq t\}}+m(1+\delta)\mathbf{1}_{\{a_i> t\}}}) \label{edgevariable}\\
        =&\p(A\mid \mathcal{F}_t)\prod\limits_{i:b_i\leq t}\mathbf{1}_{\{v_{b_i,t_i}=\V{a_i}\}}\cdot \prod\limits_{i:a_i\leq t<b_i}\frac{\wdb{a}{t}+\#\{j\mid 1\leq j<i,a_j=a_i,b_j>t\}}{(2+\delta)mt^{\frac{1}{2+\delta}}b_i^{\frac{1+\delta}{2+\delta}}}\notag \\
        &\cdot \prod\limits_{i:a_i> t}((1+\delta)m+\#\{j\mid 1\leq j<i,a_j=a_i\})b_i^{-\frac{1+\delta}{2+\delta}}(a_i-1)^{-\frac{1}{2+\delta}}.
    \end{align}
    The purpose of defining the random process $X_t$ is to create a process with properties similar to a martingale and finally becomes an indicator function of the event on the left side of \eqref{eq:3.1}. Then we can use the Optional Stopping Theorem to bound the probability of the event.
    
     When $t\notin \cup_{i=1}^{k}\{a_i,b_i\}$ and $X_{t-1}>0$, we have $\p(A\mid  \mathcal{F}_t)=\p(A\mid \mathcal{F}_{t-1})$, and for $M\leq t$,
    \begin{align*}
        &\e\left(\left. \mathbf{1}_{G_{M,t}}\frac{X_t}{X_{t-1}}\right| \mathcal{F}_{t-1}\right)\\
        \leq&\e\left(\mathbf{1}_{G_{M,t}}\prod\limits_{i:a_i\leq t<b_i}\frac{\wdb{a_i}{t}+\#\{j\mid 1\leq j<i,a_j=a_i,b_j>t\}}{\wdb{a_i}{t-1}+\#\{j\mid 1\leq j<i,a_j=a_i,b_j>t\}}\left(1-\frac{1}{(2+\delta)t}+O(t^{-2})\right)\right)
    \end{align*}
    Notice that
    \begin{align}
        &\e\left(\left.\mathbf{1}_{G_{M,t}}\prod\limits_{i:a_i\leq t<b_i}\frac{\wda{a_i}{t}{l+1}+\#\{j\mid 1\leq j<i,a_j=a_i,b_j>t\}}{\wda{a_i}{t}{l}+\#\{j\mid 1\leq j<i,a_j=a_i,b_j>t\}}\right| \mathcal{F}_{t,l}\right)\notag\\
        \leq&\mathbf{1}_{G_{M,t}}\left(1+\sum\limits_{i=1}^{t-1}\frac{\#\{j\mid 1\leq j\leq k,a_j=i,b_j>t\}}{\wda{i}{t}{l}}\frac{\mathbf{1}_{\{D(\V t,\V i)\leq r\}}\wda{i}{t}{l}}{L(t)-m+l}\right)\notag\\
        =&1+\frac{\#\{j\mid 1\leq j\leq k,a_j\leq t,b_j>t,D(\V t,\V{a_j})\leq r\}}{(2+\delta)pmt(1+O(t^{-c}))}.\label{oneedge1}
    \end{align}
    Therefore,
    \begin{align}
        \e\left(\left.\mathbf{1}_{G_{M,t}}\frac{X_t}{X_{t-1}}\right| \mathcal{F}_{t-1}\right)\leq& \e\left(\mathbf{1}_{G_{M,t}}\prod\limits_{a_i\leq t<b_i}\left(1+\frac{1}{(2+\delta)t+O(t^{1-c})}\right)\left(1-\frac{1}{(2+\delta)t}+O(t^{-2})\right)\right)\notag \\
        =&1+O(t^{-1-c}).\label{edgeincrease1}
    \end{align}
        When $a_i=t$, $t\notin \cup_{i=1}^{k}\{b_i\}$ and $X_{t-1}>0$, then $\e(\p(A\mid \mathcal{F}_{t})\mid \mathcal{F}_{t-1})=\p(A\mid \mathcal{F}_{t-1})$ and similarly \eqref{oneedge1} is true. As a result, 
        \begin{align*}
            X_t\leq \frac{\p(A\mid \mathcal{F}_t)}{\p(A\mid \mathcal{F}_{t-1})}(1+O(t^{-1}))X_{t-1}.
        \end{align*}
        By taking expectations on both sides, we have
         \begin{align}
            \e(X_t\mid \mathcal{F}_{t-1})\leq (1+O(t^{-1}))X_{t-1}. \label{edgeincrease3}
        \end{align}
    When $b_i=t$, $X_{t-1}>0$, similarly \eqref{oneedge1} is true for $l\notin \{t_i\mid b_i=t\}$. For $l=t_i$,
    \begin{align*}
        \p(v_{b_i,l}=\V{a_i}\mid \mathcal{F}_{t,l-1})=\frac{\mathbf{1}_{\{D(\V{b_i},{\V{a_i}})\leq r\}}\wda{a_i}{t}{l-1}}{L(t)-m+l-1}\leq \frac{\mathbf{1}_{\{D(\V{b_i},\V{a_i})\leq r\}}\wda{a_i}{t}{l-1}}{(2+\delta)pmt(1+O(t^{-c}))}
    \end{align*}
    on $G_{M,t}$. Notice that $\p(A\mid \mathcal{F}_t)>0$ implies $D(\V{a_i},\V{b_i})\leq r$. Therefore,
    \begin{align}
        \e\left(\left. \frac{X_t}{X_{t-1}}\right| \mathcal{F}_{t-1}\right)\leq 1+O(t^{-c}).\label{edgeincrease2}
    \end{align}
    Combining \eqref{edgeincrease1}, \eqref{edgeincrease3} and \eqref{edgeincrease2} gives
    \begin{align*}
        \e(\mathbf{1}_{G_{M,n}}X_n\mid \mathcal{F}_M)\leq X_M(1+O(M^{-c})).
    \end{align*}
    Let $M=\inf\{t\mid G_{t,n}\text{~happen}\}$. Classifying the probability in \eqref{eq:3.1} by the value of $M$, a union bound gives
    \begin{align}
        &\p(v_{b_i,t_i}=\V{a_i},\ \forall \ 1\leq i \leq k)\notag \\
        \leq &\p(v_{b_i,t_i}=\V{a_i},\ \forall\ 1\leq i \leq k, \ M\leq a')+\sum\limits_{j=a'+1}^{n}\p(v_{b_i, t_i}=\V{a_i},\ \forall\ 1 \leq i \leq k,\ M=j)\notag \\
        \leq &(1+O((a')^{-c}))X_{a'}+\sum\limits_{j=a'+1}^{n}(1+O((j)^{-c}))\e(\mathbf{1}_{\{M=j\}}X_j)\notag\\
        \leq &(1+O((a')^{-c}))X_{a'}+\sum\limits_{j=a'+1}^{n}o(\exp(-(pj)^{c})((2+\delta)mj)^{k}f(A)^{-1})X_a'  \label{eq:3.8} \\
        =&(1+O((a')^{-c}))X_{a'}=\text{RHS~of~\eqref{eq:3.1}},\notag
    \end{align}
     where \eqref{eq:3.8} holds because $\p(M=j)=o(\exp(-(pj)^{c}))$ and $\frac{\ell(a,b,t,j)}{\ell(a,b,t,a')}\leq (2+\delta)mj$ always hold. This completes the proof of the upper bound.
     
    The proof of the lower bound is similar. Define $\tau=\inf\{t\mid G_{a',t} \text{ does not happen}\}$,
    \begin{align*}
        Y_t=\mathbf{1}_{G_{a',t}}X_t+\mathbf{1}_{G_{a',t}^c}Z_{\tau}
    \end{align*}
    and 
    \begin{align*}
        Z_t=\prod\limits_{i:a_i\leq t\leq b_i}\frac{t+k}{pt^{\frac{1}{2+\delta}}b_i^{\frac{1+\delta}{2+\delta}}}\prod\limits_{i:t< a_i}\frac{m(1+\delta)+k}{(2+\delta)mp(a_i-1)^{\frac{1}{2+\delta}}b_i^{\frac{1+\delta}{2+\delta}}},
    \end{align*}
    which is greater than the maximum possible value of $X_t$. The random process $Z_t$ acts like a ``compensation" process to make $Y_n$ always be an approximately sub-martingale. The same method will be used again in Section \ref{maxdegsec}.
    Notice that $\frac{\ell(a,b,s,t)}{\ell(a,b,s,t-1)}\leq (2+\delta)mt$. Similarly to the proof in the upper bound, we have 
    \begin{align*}
        &\e\left( \left.\frac{Y_{t+1}}{Y_t} \right| \mathcal{F}_t\right)\geq 1-O(t^{-1-c}),\qquad  t\notin \cup_{i=1}^{k}\{a_i,b_i\},\\
        &\e\left(\left.\frac{Y_{t+1}}{Y_t}\right| \mathcal{F}_t\right)\geq 1-O(t^{-c}),\qquad    t\in \cup_{i=1}^{k}\{a_i,b_i\}.\\
    \end{align*}
    Notice that $Z_t\leq (t+k)^k\p(A)^{-1}Y_a'$
    \begin{align*}
        &\p(v_{b_i,t_i}=\V{a_i},\ \forall\ 1\leq i\leq k)\geq \e(X_n)=\e(Y_n)-\e(Z_{\tau}\mathbf{1}_{G^c_{a',n}})\\
        =&(1-O((a')^{-c}))Y_{a'}-\sum_{j=a'}^{n}o(\exp(-pj)^c)(j+k)^{k}\p(A)^{-1}Y_{a'}=\text{RHS of \eqref{eq:3.1}},
    \end{align*}
    which completes the proof.
\end{proof}
  \section{Number of triangles}\label{sectiontc}
  In this section, we will prove Theorem \ref{main-triangle-count}. The main idea of the proof is the second moment method. We use Theorem \ref{thm3.1} to bound the first and second moments of the number of triangles. 
  
  Let $T_n$ be the number of triangles in $GPM_n$, recall that 
  \begin{equation*}
      F_p=\e(f(\V i,\V j)\mid \V j\in B(\V i,r))=\p(D(\V i,\V j)\leq r\mid \V i,\V j\sim Unif(B(\V i,r))).
  \end{equation*}
\begin{proposition}
    \begin{align*}
    \e(T_n)= \frac{m(m-1)(1+\delta)^2(m(1+\delta)+1)F_p \log(n)}{(2+\delta)\delta^2p}+O(1).
\end{align*}
\end{proposition}

\begin{proof}
Notice that
    \begin{align*}
        T_n=\sum\limits_{\substack{a<b<c\\1\leq t_1,t_2,t_3\leq m}}\mathbf{1}_{\{v_{b,t_1}=\V a\}}\mathbf{1}_{\{v_{c,t_2}=\V a\}}\mathbf{1}_{\{v_{c,t_3}=\V b\}}.
    \end{align*}
For each term in the summation, take $k=3$, $a_1=a_2=a$, $a_3=b_1=b$, $b_2=b_3=c$ in Theorem \ref{thm3.1}. Then we have
\begin{align*}
    \e(T_n)=m^2(m-1)\frac{F_p}{p}\sum\limits_{1\leq r<s<t\leq n}(1+O(r^{-c}))\frac{(1+\delta)^2m^2((1+\delta)m+1)}{(2+\delta)^3m^3}r^{-\frac{2}{2+\delta}}s^{-1}t^{-\frac{2+2\delta}{2+\delta}}.
\end{align*}
By some straightforward calculation,
\begin{align}
    \sum\limits_{1\leq r<s<t\leq n}r^{-\frac{2}{2+\delta}}s^{-1}t^{-\frac{2+2\delta}{2+\delta}}=&\sum\limits_{1\leq s<t\leq n}\frac{2+\delta}{\delta}(s^{\frac{\delta}{2+\delta}}-O(1))s^{-1}t^{-\frac{2+2\delta}{2+\delta}}\notag \\
    =&\sum\limits_{1\leq t\leq n}\frac{2+\delta}{\delta}(\frac{2+\delta}{\delta}t^{\frac{\delta}{2+\delta}}-O(\log(t)))t^{-\frac{2+2\delta}{2+\delta}}\notag \\
    =&(\frac{2+\delta}{\delta})^2\log(n)+O(1)\label{counttriangle}.
\end{align}
And
\begin{align*}
    &\sum\limits_{1\leq r<s<t\leq n}r^{-\frac{2}{2+\delta}-c}s^{-1}t^{-\frac{2+2\delta}{2+\delta}}=\sum\limits_{1\leq s<t\leq n}O(s^{-\frac{2}{2+\delta}-c}t^{-\frac{2+2\delta}{2+\delta}})\\
    =&\sum\limits_{1\leq t\leq n}O(t^{-1-c}) =O(1).
\end{align*}
Then we have
\begin{align*}
    \e(T_n)= \frac{m(m-1)(1+\delta)^2(m(1+\delta)+1)F_p \log(n)}{(2+\delta)\delta^2p}+O(1).
\end{align*}
\end{proof}
\begin{proposition}
    \begin{align*}
        \e(T_n^2)-\e(T_n)^2=O(\log(n)).
    \end{align*}
\end{proposition}
\begin{proof}
    Similarly to the proof in the first moment,
    \begin{align*}
        T_n^2=\sum\limits_{\substack{a<b<c\\d<e<f\\1\leq t_i\leq m}}\mathbf{1}_{\{v_{b,t_1}=\V a\}}\mathbf{1}_{\{v_{c,t_2}=\V a\}}\mathbf{1}_{\{v_{c,t_3}=\V b\}}\mathbf{1}_{\{v_{e,t_4}=\V d\}}\mathbf{1}_{\{v_{f,t_5}=\V d\}}\mathbf{1}_{\{v_{f,t_6}=\V e\}}
    \end{align*}
    Denote the right side by $\mathbf{1}_{\Delta}$. Let $X=\{a<b<c,d<e<f,1\leq t_i\leq m\}$, $A_i=X\cap\{|\{a,b,c\}\cap\{d,e,f\}|=i\}$, then similarly to \eqref{counttriangle}, by using Theorem \ref{thm3.1} for each possible shape that can be formed by combining of two triangles, we have
    \begin{align*}
        &\e(\sum_{A_0}\mathbf{1}_\Delta)\leq (\e(T_n)+O(1))^2=\e(T_n)+O(\log(n)),\\
        &\e(\sum_{A_1}\mathbf{1}_\Delta)=O(1),\\
        &\e(\sum_{A_2}\mathbf{1}_\Delta)=O(1),\\
        &\e(\sum_{A_3}\mathbf{1}_\Delta)=O(\log(n)).
    \end{align*}
    By adding them together, we finish the proof.
\end{proof}
Now we present the proof of Theorem \ref{main-triangle-count}. Notice that $T_n$ is a non-decreasing process. For any constants $M,\varepsilon>0$, take $a_i=\text{e}^{(1+\frac{\varepsilon}{2})^i\log(M)}$, then
\begin{align*}
    \p(|T_{a_i}-\e(T_{a_i})|\geq \frac{\varepsilon}{2}\e(T_{a_i}))\leq O(\varepsilon^{-2} \log(a_i)^{-1})),
\end{align*}
and
\begin{align}
    \p(\exists \ n\geq M, |T_{n}-\e(T_{n})|\geq \varepsilon\e(T_{n}))\leq &\sum_{i=0}^{\infty}\p(|T_{a_i}-\e(T_{a_i})|\geq \frac{\varepsilon}{2}\e(T_{a_i}))\notag\\
    &\leq O(\varepsilon^{-3}\log(M)^{-1}).\label{eq:triangle-count-2}
\end{align}
The right side of \eqref{eq:triangle-count-2} tends to $0$ as $M\rightarrow \infty$, giving Theorem \ref{main-triangle-count}.

\section{Maximum degree}\label{maxdegsec}
 Before proving Theorem \ref{main-maximum-degree}, we need the following modifications of Lemmas \ref{me} and \ref{me2}.
\begin{lemma}\label{me3}
    For constants $C,c_n\geq 0$, let $X_n$ be any random process satisfying 
    \begin{align*}
        &\e(X_{n+1})\leq (1+c_n)X_n,\\
        &X_{n+1}=X_n \text{~or~} X_n+1\leq X_{n+1}\leq X_n+C.
    \end{align*}
    Then 
    \begin{align*}
        \p\left(\left.\exists \ n\geq N, X_n\geq M\prod\limits_{i=N}^{n-1}(1+c_i) \right|\mathcal{F}_{X_N}\right)\leq \exp\left(-\frac{(M-X_N)^2}{8MC^2}\right).
    \end{align*}
\end{lemma}
\begin{proof}
    Define $Y_n=\prod\limits_{i=N}^{n-1}(1+c_i)^{-1}X_n$, $\tau=\inf \{s\geq N\mid Y_s\geq M\}$, then $Y_n$ is a supermartingale, and for $n\leq \tau$, $\theta\leq C^{-1}$, 
    \begin{align*}
        \e(\exp(\theta(Y_{n+1}-Y_n))\mid \mathcal{F}_{X_n})\leq 1+2\theta^2MC^2c_n\prod\limits_{i=N}^{n}(1+c_i)^{-1}.
    \end{align*}
    Then
    \begin{align}
        \e(\exp(\theta Y_n))\leq \exp(2\theta^2MC^2+\theta Y_N).\label{eqme3}
    \end{align}
    Taking $\theta=\frac{M-Y_N}{4MC^2}$ in \eqref{eqme3} completes the proof of the lemma.
\end{proof}
Similarly, for the upper bound, we have
\begin{lemma}\label{me4}
     For constants $C,c_n\geq 0$, let $X_n$ be any random process satisfying 
    \begin{align*}
        &\e(X_{n+1})\geq (1+c_n)X_n,\\
        &X_{n+1}=X_n \text{~or~} X_n+1\leq X_{n+1}\leq X_n+C.
    \end{align*}
    Then for $M\leq X_N-8C^2$, 
    \begin{align*}
        \p\left(\left.\exists \ n\geq N,\ X_n\leq M\prod\limits_{i=N}^{n-1}(1+c_i) \right|\mathcal{F}_{X_N}\right)\leq 2\exp\left(-\frac{(M-X_N)^2}{16X_NC^2}\right).
    \end{align*}
\end{lemma}
\begin{proof}
    Define $Y_n=\prod\limits_{i=N}^{n-1}(1+c_i)^{-1}X_n$,  $\tau=\inf \{s\geq N\mid Y_s\leq M\}$, then $Y_n$ is a submartingale, and for $n\leq \tau$, $\theta\leq C^{-1}$, 
    \begin{align*}
        \e(-\theta(Y_{n+1}-Y_n)\mid \mathcal{F}_{X_n})\leq 1+2\theta^2Y_nC^2c_n\prod\limits_{i=N}^{n}(1+c_i)^{-1}.
    \end{align*}
    Then for $Y_N^{-1}\leq \theta \leq C^{-1}$, 
    \begin{align*}
        &\e(\exp(-\theta Y_{n+1})+\exp(-\theta Y_N)\mid \mathcal{F}_{X_n})\\
        \leq& \left(1+2\theta^2Y_nC^2\prod\limits_{i=N}^{n}(1+c_i)^{-1}c_n\right)\exp(-\theta Y_n)+\exp(-\theta Y_N)\\
        \leq &\left(1+4\theta^2Y_NC^2\prod\limits_{i=N}^{n}(1+c_i)^{-1}c_n\right)(\exp(-\theta Y_{n})+\exp(-\theta Y_N)).
    \end{align*}
    Then
    \begin{align}
        \e(\exp(-\theta Y_n)+\exp(-\theta Y_N))\leq 2\exp(-\theta Y_N+4\theta^2C^2Y_N). \label{eqme4}
    \end{align}
    Take $\theta=\frac{Y_N-M}{8Y_NC^2}$ in \eqref{eqme4} to complete the proof of the lemma.
\end{proof}
Now we start proving (1) in Theorem \ref{main-maximum-degree}. First, we will prove the upper bound, that is, with high probability, the maximum degree in the network will not be too large.

Suppose $\degc{i}{n}$ is the degree of vertex $i$ in RGG at time $n$. Taking the constants $1<C_3<C$ sufficiently large, $B=\lfloor Cp^{-1}\log(p^{-1})\rfloor$. Since for each other vertex, the probability that it connects to $i$ is $p$ independently. Then, by the large deviation of i.i.d. random variables, we have the following inequality,
\begin{align}
    \p\left(\degc{i}{B} \geq 5C\log(p^{-1})\right)\leq \exp\left(-C\log(p^{-1})\right).\label{maxdegstes}
\end{align}
Define events 
\begin{equation}
\begin{split}
A := \bigcup_{i=B+1}^{\infty} \left\{(2+\delta)mpi-m(\log(p^{-1}))^{\frac{1}{C_3}}(pi)^{\frac{C_3-1}{C_3}} \leq L(i) \right. \\
\left. \leq (2+\delta)mpi+m(\log(p^{-1}))^{\frac{1}{C_3}}(pi)^{\frac{C_3-1}{C_3}}-m \right\}.
\end{split}\label{maxdegAdef}
\end{equation}
Then by Theorem \ref{thm2.10}
\begin{align*}
    \p(A^c)&\leq \sum\limits_{i=\lceil Cp^{-1}\log(p^{-1}) \rceil}^{\infty}C_1\exp\left(-C_2\log(p^{-1})^{\frac{\alpha}{C_3}}(pi)^{\frac{C_3-\alpha}{C_3}}\right)\\
    &\leq C_1\exp\left(-C_2\log(p^{-1})C^{\frac{C_3-\alpha}{C_3}}+O(\log(p^{-1}))\right).
\end{align*}
Define
\begin{equation*}
    c_n:=\frac{p}{(2+\delta)pn-\left(\log(p^{-1})\right)^{\frac{1}{C_3}}(pn)^{\frac{C_3-1}{C_3}}}.
\end{equation*}
Recall that $\wdb{i}{n}$ is the weight of the vertex $\V i$ at time $n$, which is $\dega{i}{n}+m\delta$. Then on the event $A$, 
\begin{align*}
    &\e(\wdb{i}{n+1}\mid \mathcal{F}_n)\leq (1+c_n)\wdb{i}{n},\\
    &\wdb{i}{n+1}=\wdb{i}{n} \text{~or~} \wdb{i}{n}+1\leq \wdb{i}{n+1}\leq \wdb{i}{n}+m.
\end{align*}
Notice that $\prod\limits_{i=B}^{n}(1+c_i)=\left(1+O\left(C^{\frac{1}{C_3}}\right)\right)\left(\frac{n}{B}\right)^{\frac{1}{2+\delta}}$, then by Lemma \ref{me3} and \eqref{maxdegstes}, for $i\leq B$, 
\begin{align*}
    \p\left(\wdb{i}{n}\geq 10C\log(p^{-1})\left(\frac{n}{B}\right)^{\frac{1}{2+\delta}}\right)\leq \exp\left(-\frac{C\log(p^{-1})}{m^2}\right)+\exp\left(-C\log(p^{-1})\right).
\end{align*}
For $i>B$, using Lemma \ref{me3} again, we have
\begin{align*}
    \p\left(\wdb{i}{n}\geq 20C\log(p^{-1})\left(\frac{n}{B}\right)^{\frac{1}{2+\delta}}\right)\leq \exp\left(-\frac{C(\frac{i}{B})^{\frac{1}{2+\delta}}\log(p^{-1})}{m^2}\right).
\end{align*}
Taking a union bound, we can get the following upper bound:
\begin{align}
    &\p\left(\max\limits_{1\leq i\leq n}(\dega{i}{n})\geq 10C^{\frac{1+\delta}{2+\delta}}\log(p^{-1})^{\frac{1+\delta}{2+\delta}}(np)^{\frac{1}{2+\delta}}\right)\notag\\
    \leq &2Cp^{-1}\log(p^{-1})\exp\left(-\frac{C\log(p^{-1})}{m^2}\right)\notag \\
    +&\sum\limits_{i=
    \lceil Cp^{-1}\log(p^{-1})\rceil }^{\infty}\exp\left(-\frac{C^{\frac{1+\delta}{2+\delta}}i^{\frac{1}{2+\delta}}p^{\frac{1}{2+\delta}}\log(p^{-1})^{\frac{1+\delta}{2+\delta}}}{m^2}\right).\label{upmaxdeg}
\end{align}

Note that the RHS in \eqref{upmaxdeg} converges to $0$ when $C\rightarrow \infty$ or $C> m^2, p\rightarrow 0$. Therefore, we complete the proof of the upper bound of (1) in Theorem \ref{main-maximum-degree}.


 To prove the lower bound, we need the following lemma.
 \begin{lemma}\label{lem5.4}
 For any $C_1,\varepsilon>0$, there exists $C_2>0$ such that for any $p\in (0,1)$, when $B\geq C_2p^{-1}\log(p^{-1})$, then 
     \begin{align*}
         \p(\exists \ 1\leq i \leq B,\dega{i}{B}\geq C_1\log(p^{-1}))\geq 1-\varepsilon p^{10}.
     \end{align*}
 \end{lemma}
\begin{proof}
     For the unit area ball $S\subset \mathbb{R}^d$, we can put $s\geq\lfloor (3^d p)^{-1}\rfloor$ non-intersecting balls with radius $\frac{3}{2}r$ noted by $B_{x_1},B_{x_2}\dots, B_{x_{s}}$. We can assume the vertices are generated by a Poisson point process with intensity $M$, then the vertices distribution in these balls are independent, suppose the number of vertices in this process is $M'$, then
     \begin{align*}
         \p(M'\leq 2M)\geq 1-exp\left(-\frac{M}{10}\right).
     \end{align*}
    Recall that $\gV{n}=\{ \V 1,\V 2\dots, \V n\} $. In each ball $B_x=B(x,\frac{3}{2}r)$, we say that the ball is good if there are at least $\lceil (\frac{3}{2})^{d-1}pM\rceil$ vertices in $B_x$, and
    \begin{align*}
         \#\left\{\gV{i}\cap B\left(x,\frac{1}{2}r\right)\right\}\geq 3^{-d-1}\#\left\{\gV{i}\cap B\left(x,\frac{3}{2}r\right)\right\}
     \end{align*}for every $i>0$. Suppose that the vertices in $B(x,\frac{1}{2}r)$ are $\V {b_1},\V {b_2} \dots, \V {b_{t}}$. Then 
    \begin{align*}
        \p\left(B\left(x,\frac{3}{2}r\right) \text{ is good}\right)\geq 3^{-d-10}
    \end{align*}
    and when the ball is good, 
    \begin{equation}
    \begin{split}  
        &\p\left(\text{all $m$ edges from } \V {b_i} \text{ connect to } \V{b_1},\ \forall \ 1\leq i\leq \frac{\log(p^{-1})}{100dm}\right)\\
        \geq &\left(3^{-(d+1)}(2+\delta)\right)^{\frac{\log(p^{-1})}{100d}}\geq p^{0.1}\label{maxdeggoodball}
    \end{split} 
    \end{equation}
    When the event in LHS of \eqref{maxdeggoodball} is hold, for $i\geq \frac{\log(p^{-1})}{100dm}$, 
    \begin{align*}
        \p((\V{b_i},\V{b_1})\in GPM_{b_i}\mid \mathcal{F}_{b_i-1})\geq \frac{\log(p^{-1})}{100d*3^di}
    \end{align*}
    Take $M\geq \frac{2^{d-1}\log(p^{-1})}{100dpm3^{d-1}}\exp(200d\times 3^dC_1)$. Then for each good ball, conditioned on the event in LHS of \eqref{maxdeggoodball} holds, 
    \begin{align*}
        \sum\limits_{i=\lceil \frac{\log(p^{-1})}{100dm}\rceil}^{B}\p((\V{b_i},\V{b_1})\in GPM_{b_i}\mid \mathcal{F}_{b_i-1})\geq \frac{3}{2}C_1\log(p^{-1}).
    \end{align*}
    By Markov's inequality, we have
    \begin{align*}
        \p(\dega{b_1}{M'}\geq C_1\log(p^{-1}))\geq \frac{1}{3}.
    \end{align*}
    Then with probability at least $3^{-d-11}p^{0.1}$ there exists a vertex $\V i$ in this ball such that $\dega{i}{M'}\geq C_1\log(p^{-1})$. Since there are at least  $(3^d p)^{-1}$ balls, take $B=2M$,
    \begin{align}
        &\p\left(\exists \ 1\leq i \leq B,\ \dega{i}{B}\geq C_1\log(p^{-1})\right)\notag\\
        \geq&\p\left(\exists \ 1\leq i\leq s, \V j\in B_i,\dega{j}{M'}\geq C_1\log(p^{-1})\right)-\p(M'\geq 2M)\notag\\
        \geq &1-\left(1-3^{-d-11}p^{0.1}\right)^{s}-\exp\left(-\frac{M}{10}\right)
        \geq 1-\exp\left(3^{-11}p^{-0.9}\right)-\exp\left(-\frac{M}{10}\right).\label{maxdeglwbnd}
    \end{align}
     For $p$ sufficiently small, \eqref{maxdeglwbnd} is larger than $1-\varepsilon p^{10}$, and for the other $p$, since $C_1\log(p^{-1})=O(1)$, and almost surely, $\lim\limits_{n\rightarrow \infty}\dega{1}{n}=\infty$, we only need to take $C_2$ sufficiently large. 
\end{proof}
Following the proof of the upper bound, let $A$ be the same event as in \eqref{maxdegAdef}, define 
\begin{equation*}
    c_n:=\frac{p}{(2+\delta)pn+(\log(p^{-1}))^{\frac{1}{C}}(pi)^{\frac{C-1}{C}}}.
\end{equation*}
Then conditioned on $A$, 
\begin{align*}
    &\e(\wdb{i}{n+1}\mid \mathcal{F}_n)\geq (1+c_n)\wdb{i}{n},\\
    &\wdb{i}{n+1}= \wdb{i}{n} \text{~or~} \wdb{i}{n}+1\leq \wdb{i}{n+1}\leq \wdb{i}{n}+m.
\end{align*}
Take $i$ in Lemma \ref{lem5.4} satisfying $\dega{i}{B}\geq C_1\log(p^{-1})$, let $X_n=\wdb{i}{n}$, $C=m$ in Lemma \ref{me4}, then we have
\begin{align}
    \p\left(\wdb{i}{n}\leq \frac{C_1}{2}\log(p^{-1})\left(\frac{n}{B}\right)^{\frac{1}{2+\delta}}\right)\leq 2\exp\left(-\frac{C_1\log(p^{-1})}{64m^2}\right).\label{maxdegfinallowerbnd}
\end{align}
Notice that the RHS in \eqref{maxdegfinallowerbnd} converges to $0$ when $C_1\rightarrow \infty$ or $p\rightarrow 0$. Then we finish the proof of the lower bound of (1) in Theorem \ref{main-maximum-degree}.
Combining the lower bound and the upper bound gives (1) in Theorem \ref{main-maximum-degree}.

Now we start proving (2) in Theorem \ref{main-maximum-degree}. Denote $X_n=\max_{1\leq i \leq n}(\wdb{i}{n})$, for any $N> 0$, by lemma \ref{me3}, 
\begin{align*}
    &\p\left(\exists \ n\geq N, X_n\geq (1+\varepsilon)X_N\left(\frac{n}{N}\right)^{\frac{1}{2+\delta}}\right)\\
    \leq &\sum_{i=1}^{\infty}\p\left(\exists \ n\geq N, \wdb{i}{n}\geq (1+\varepsilon)X_N\left(\frac{n}{N}\right)^{\frac{1}{2+\delta}}\right)\\
    \leq& N\exp\left(-\frac{\varepsilon^2X_N}{8m^2(1+\varepsilon)}\right)+\sum\limits_{i=N}^{\infty}\exp\left(-\frac{X_N\left(\frac{i}{N}\right)^{\frac{1}{2+\delta}}}{16m^2}\right).
\end{align*}
The right side converges to $0$ when $N\rightarrow \infty$. 

Similarly, according to Lemma \ref{me4},
\begin{align*}
    \p\left(\exists \ n\geq N, X_n\leq (1-\varepsilon)X_N\left(\frac{n}{N}\right)^{\frac{1}{2+\delta}}\right)\leq \exp\left(-\frac{\varepsilon^2X_N}{16m^2}\right)
\end{align*}
which also converges to $0$. 
As a result, almost surely $A_n=\frac{X_n}{\log(p^{-1})^{\frac{1+\delta}{2+\delta}}(np)^{\frac{1}{2+\delta}}}$ is a Cauchy sequence. This gives (2) in Theorem \ref{main-maximum-degree}.
\begin{conjecture}
When $p \rightarrow 0$, $X$ converges in distribution to a constant random variable. Or, equivalently, we can take $D_1=D_2$ in \eqref{eq:thm1.2-b}.
\end{conjecture}
An intuitive explanation for this conjecture is that for each vertex $\V i$, the distribution of $\dega{i}{n}$ is similar to the degree distribution of the $\lfloor pi\rfloor$th vertex in a PAM with $pn$ vertices. So, the maximum degree of the GPM is similar to the maximum degree of $p^{-1}$ independent PAM with $pn$ vertices. So we guess
\begin{align*}
    \e(\#\{1\leq i\leq n\mid (\dega{i}{n})>C\log(p^{-1})^{\frac{1+\delta}{2+\delta}}(np)^{\frac{1}{2+\delta}}\})=p^{f(C)-1+o(1)}.
\end{align*}
And $X$ converges to the solution of $f(x)=1$ when $p\rightarrow 0$. However, the main difficulty here is that the degrees of each vertices are not really independent, and it is hard to analyze how the degree grows earlier in the process.

\section{Connectivity}\label{sectionconnect}

In this section, we will prove Theorem \ref{main-connectivity}. The main idea of the proof comes from \cite{Flaxman01012006}, we will discuss the change in the number of connected components when adding a new vertex. However, we will make a more careful estimate to give a precise result of the connecting time. Moreover, we found that the main difficulty of connectivity comes from isolated vertices, so we will add some weight to the number of isolating vertices. 
\subsection{Connected regime}
 Define $C_n$ as the number of connected components in $GPM_n$. $I_n$ is the number of isolated vertices in $GPM_i$, let $X_n=C_n+0.1I_n$ and $M$ is a large constant.
 \begin{lemma}\label{lem6.2}
  For $n>\frac{M}{p}$,
     \begin{align*}
         \e(X_{n+1})\leq \left(1-\frac{m-0.99}{n}\right)(\e(X_n)-1)+\frac{M}{(pn)^m}+1.
     \end{align*}
 \end{lemma}
\begin{proof}
The main idea of this lemma is that for each connected component, when a new vertex is added near the component, it creates an opportunity for the component to be connected to the new vertex, thereby reducing its weight. By attributing the total weight reduction to individual connected components, we can show that every component except the largest one contributes at least $\frac{m-0.99}{n}$ to the overall weight reduction.

When $\V{n+1}$ is added to the graph, we analyze the change in the number of connected components. Suppose that the connected components of $GPM_n$ are $A_1, A_2, \dots, A_t$, where $t = C_n$ denotes the number of components. For the newly added vertex $\V{n+1}$, we define
     \begin{align*}
         x_i:=\sum\limits_{\V j\in A_i\cap B(
         \V{n+1},r)}\wdb{j}{n}.
     \end{align*}
     Then $\sum\limits_{i=1}^{t}x_i=L(n+1)-m(2+\delta)$. For each component $A_i$,
     \begin{equation*}
         \p(\V{n+1} \text{~connect~to~} A_i)\geq 1-\left(1-\frac{x_i}{L(n+1)}\right)^{m}
     \end{equation*}
    which gives 
     \begin{align*}
         \e(X_{n+1})-X_n\leq 1-\sum\limits_{i=1}^{t}(1+0.1*\mathbf{1}_{\{|A_i|=1\}})\left(1-\left(1-\frac{x_i}{L(n+1)}\right)^{m}\right)+0.1\left(\frac{2+\delta}{L(n+1)}\right)^m.
     \end{align*}

     Let $c$ be the constant in Theorem \ref{thm2.10}. Conditioned on $GPM_n$ and the position of $\V{n+1}$, define
     \begin{align*}
         &a_n=\left(\frac{2(m-1)}{pn}L(n+1)+2\right)\mathbf{1}_{\{L(n+1)\notin[(2+\delta)mpn-(pn)^{1-c},(2+\delta)mpn+(pn)^{1-c}]\}},\\
         &b_n=\frac{200(m-1)}{pn}L(n+1)\mathbf{1}_{\{\#\{1\leq i\leq n\mid D(\V i,\V{n+1})\leq 100^{-\frac{1}{d}}r\}\leq \frac{pn}{200}\}},\\
         &c_n=\sum\limits_{i=1}^{t}-\frac{((1+0.1*\mathbf{1}_{\{|A_i|=1\}})m-1.01)\min(x_i,100m(2+\delta))\mathbf{1}_{\{x_i\leq L(n+1)-200m(2+\delta)\}}}{(2+\delta)mpn+2(pn)^{1-c}},\\
         &d_n=2\left(\frac{2+\delta}{(2+\delta)mpn-(pn)^{1-c}}\right)^{m}.
     \end{align*}
     Here $c_n$ is the main contribution to the reduction of weight, $a_n$ and $b_n$ compensate for some rare events, and $d_n$ is to counteract the increase in weight when the new vertex only connects to itself. Then for $n>\frac{M}{p}$,
     \begin{align*}
         &\e(X_{n+1})-X_n\\
         \leq&-\sum\limits_{i=1}^{t}\left((1+0.1*\mathbf{1}_{\{|A_i|=1\}})\left(1-\left(1-\frac{x_i}{L(n+1)}\right)^{m}\right)-\frac{x_i}{L(n+1)-m(2+\delta)}\right)\\
         &+0.1\left(\frac{2+\delta}{L(n+1)}\right)^m\\
         \leq& -\sum\limits_{i=1}^{t}\left(\frac{((1+0.1*\mathbf{1}_{\{|A_i|=1\}})m-1.01)\min(x_i,100m(2+\delta))\mathbf{1}_{\{x_i\leq L(n+1)-200m(2+\delta)\}}}{L(n+1)}\right)\\
         &\cdot \mathbf{1}_{\{L(n+1)\in[(2+\delta)mpn-(pn)^{1-c},(2+\delta)mpn+(pn)^{1-c}]\}}+1.1\left(\frac{2+\delta}{L(n+1)}\right)^m\\
         \leq& c_n+d_n+(-c_n+2)\mathbf{1}_{\{L(n+1)\notin[(2+\delta)mpn-(pn)^{1-c},(2+\delta)mpn+(pn)^{1-c}]\}}\\
         \leq& a_n+c_n+d_n.
     \end{align*}
     Define
     \begin{align*}
         &a_n^i=\frac{2(m-1)}{pn}x_i\mathbf{1}_{\{L(n+1)\notin[(2+\delta)mpn-(pn)^{1-c},(2+\delta)mpn+(pn)^{1-c}]\}},\\
         &b_n^i=\frac{200(m-1)}{pn}x_i\mathbf{1}_{\{\#\{1\leq i\leq n\mid D(\V i,\V{n+1})\leq 100^{-\frac{1}{d}}r\}\leq \frac{pn}{200}\}},\\
         &c_n^i=-\frac{((1+0.1*\mathbf{1}_{\{|A_i|=1\}})m-1.01)\min(x_i,100m(2+\delta))\mathbf{1}_{\{x_i\leq L(n+1)-200m(2+\delta)\}}}{(2+\delta)mpn+2(pn)^{1-c}}.\\
     \end{align*}
     Here, $a_n^i$, $b_n^i$ and $c_n^i$ divide $a_n$, $b_n$ and $c_n$ into each component. Then 
     \begin{align*}
         \sum\limits_{i=1}^{t}a_n^i\leq a_n,\quad \sum\limits_{i=1}^{t}b_n^i\leq b_n,\quad \sum\limits_{i=1}^{t}c_n^i= c_n.
     \end{align*}
    We use Lemma \ref{lem6.2a}, proved later, to complete the proof. Since $\e(a_n),\e(b_n),\e(d_n)\leq \frac{2}{(pn)^m}$, by Lemma \ref{lem6.2a}, we have 
         \begin{align*}
             \e(X_{n+1}-X_n)&\leq \e(a_n+c_n+d_n)\leq \sum_{i=1}^{t}\e(c_n^i-a_n^i-b_n^i)+2a_n+b_n+d_n\\
             &\leq \e\left(-(X_n-1)\left(\frac{m-0.99}{n}\right)+2a_n+b_n+d_n\right).
         \end{align*}
         Then 
         \begin{align*}
             \e(X_{n+1}-1)\leq \left(1-\frac{m-0.99}{n}\right)(\e(X_n)-1)+\frac{M}{(pn)^m},
         \end{align*}
         which completes the proof of Lemma \ref{lem6.2}.
\end{proof}
         
         \begin{lemma}\label{lem6.2a}         
     For any component $A_i$ that satisfies $1\leq |A_i|\leq 0.6n$, conditioned on $GPM_n$, we have
         \begin{equation}
             \e(c_n^i-a_n^i-b_n^i)\leq -\frac{\left(1+0.1*\mathbf{1}_{\{|A_i|=1\}}\right)(m-0.99)}{n}.\label{eq:lem6.2}
         \end{equation}
     \end{lemma}
     \begin{proof}
         For a component $A_i$, we define the dense region $\mathbb{D}_i$ and the sparse region $\mathbb{S}_i$ as follows:
         \begin{align*}
             &\mathbb{D}_i=\left\{x\in S\mid \sum\limits_{j\in A_i,D(x,\V j)\leq 100^{-\frac{1}{d}}r}\wdb{j}{n} \geq 100m(2+\delta)\right\},\\
             &\mathbb{S}_i=\left\{x\in S\mid 0<\sum\limits_{j\in A_i,D(x,\V j)\leq 100^{-\frac{1}{d}}r}\wdb{j}{n} < 100m(2+\delta)\right\}.\\
         \end{align*}
         Then when $\V{n+1}\in \mathbb{D}_i$, we have $x_i\geq 100m(2+\delta)$, when $\V{n+1}\in \mathbb{S}_i$, we have $x_i\geq m(1+\delta)$.
         We classify by the area of $\mathbb{D}_i$.
         \begin{enumerate}
         
         \item $A(\mathbb{D}_i) > 0.9$.
         Since $\e(x_i)= (2+\delta)mp|A_i|\leq 0.6(2+\delta)mpn$. Then $\p(\V{n+1}\in \mathbb{D}_i, x_i\leq 0.9(2+\delta)mpn)\geq 0.2$. Notice that 
         \begin{align*}
             \e(c_{n}^{i}-a_n^i)&\leq -\e\left(\frac{(m-1.01)\min(x_i,100m(2+\delta))\mathbf{1}_{\{x_i\leq 0.9m(2+\delta)pn\}}}{(2+\delta)mpn+2(pn)^{1-c}}\right)\\
             &\leq -0.2\frac{(m-1.01)100m(2+\delta)}{(2+\delta)mpn+2(pn)^{1-c}}\leq -\frac{m-0.99}{n}
         \end{align*}
         which satisfy \eqref{eq:lem6.2}.
         
         \item $\frac{p}{200}\leq A(\mathbb{D}_i)\leq 0.9$.
         Define $\mathbb{D}'_i=\{x\in S\setminus \mathbb{D}_i\mid D(x,\mathbb{D}_i)\leq (1-100^{-\frac{1}{d}})r\}$, then when $\V{n+1}\in\mathbb{D}'_i$, we still have $x_i\geq 100m(2+\delta)$. Notice that
         \begin{align*}
             L(n+1)-x_i\geq \#\{1\leq j\leq n\mid D(\V j,\V{n+1})\leq 100^{-\frac{1}d}r\}-100m(2+\delta)
         \end{align*}
         and by isoperimetric inequality, $A(D'_i)\geq (1-100^{-\frac{1}{d}}+200^{-\frac{1}{d}})^dp\geq \frac{1}{2}p$. As a result, when $b_n^i\leq \frac{100(m-1.01)}{pn}$, then $b_n^i=0$, and $\#\{1\leq i\leq n\mid D(\V i,\V{n+1})\leq 100^{-\frac{1}{d}}r\}\geq \frac{pn}{200}$. Since $pn>M$ is sufficiently large, we have 
         \begin{align*}
             L(n+1)-x_i\geq \#\{1\leq j\leq n\mid D(\V j,\V{n+1})\leq 100^{-\frac{1}d}r\}-100m(2+\delta)\geq 200m(2+\delta).
         \end{align*}
         In conclusion, we have
         \begin{align*}
             \e(c_n^i-b_n^i)\leq -\frac{10(m-1.01)}{n}\leq \frac{m-0.99}{n}.
         \end{align*}

         \item$A(\mathbb{D}_i)\leq \frac{p}{200}$.
        When $|A_i|\geq 0.9pn$, for every $x\in A_i$, we have 
        \begin{align*}
            A\left(\mathbb{S}_i\cap B(x,100^{-\frac{1}{d}}r)\right)=A\left(B(x,100^{-\frac{1}{d}}r)\right)-A\left(\mathbb{D}_i\cap B(x,100^{-\frac{1}{d}}r)\right)\geq\frac{p}{200}.
        \end{align*} 
        Then 
         \begin{align*}
             A(\mathbb{S}_i)&\geq \e_{x \sim Unif(S)}\left(\frac{1}{100m(2+\delta)}\sum\limits_{j\in A_i,D(x,\V j)\leq 100^{-\frac{1}{d}}r}\wdb{j}{n}\mathbf{1}_{\{x\in \mathbb{S}_i\}} \right)\\
             &\geq \frac{p}{200}\sum\limits_{j\in A_i,D(x,\V j)\leq 100^{-\frac{1}{d}}r}\frac{\wdb{j}{n}}{100m(2+\delta)}\geq \frac{p^2n}{30000}>\frac{pM}{30000}.
         \end{align*}
         Then, when $M$ is sufficiently large, we have
         \begin{align*}
             \e(c_n^i-b_n^i)\leq -\e\left(\frac{(m-1.01)m(1+\delta)}{(2+\delta)mpn}(1-10(pn)^{-c})\mathbf{1}_{\{x\in \mathbb{S}_i\}}\right)\leq -\frac{m-0.99}{n}.
         \end{align*} 

         When $2\leq |A_i|\leq 0.9pn$, then
         \begin{align*}
             &\e(c_n^i-a_n^i)\leq -\e\left(\frac{(m-1.01)\min(x_i,100m(2+\delta))}{(2+\delta)mpn+(pn)^{1-c}}\right)\\
             \leq& -\frac{2(m-1.01)}{n}(1-10(pn)^{-c})\leq -\frac{m-0.99}{n}.
         \end{align*}

         When $|A_i|=1$, then 
         \begin{align*}
             \e(c_n^i-a_n^i)\leq -\e\left(\frac{((1.1m-1.01)x_i}{(2+\delta)mpn+(pn)^{1-c}}\right)\leq -\frac{1.1(m-0.99)}{n}.
         \end{align*}
         \end{enumerate}
    \end{proof}

         By Lemma \ref{lem6.2}, we can give an upper bound for the number of connected components, now we only need to prove that $\e(X_n)=1+o(1)$ when $np^{\frac{m}{m-1}}=\omega(1)$.\\
         Since $\e(X_{\lceil Mp^{-1}\rceil})\leq \lceil Mp^{-1}\rceil$, we have 
         \begin{align*}
             \e(X_{n})
             \leq& 1+\exp\left(\sum\limits_{i=\lceil Mp^{-1}\rceil}^{n-1}-\frac{m-0.99}{i}\right)\lceil Mp^{-1}\rceil+\sum\limits_{i=\lceil Mp^{-1}\rceil}^{n-1}\frac{M}{(pi)^m}\exp\left(\sum\limits_{j=i+1}^{n-1}-\frac{m-0.99}{i}\right)\\
             \leq& 1+\exp(-(m-0.99)\log(npM^{-1})+O(1))\lceil Mp^{-1}\rceil\\
             +&\sum\limits_{i=\lceil Mp^{-1}\rceil}^{n-1}\frac{M}{(pi)^m}\exp(-(m-0.99)\log(i^{-1}n)+O(1))\\
             \leq& 1+O(1)(n^{-(m-0.99)}p^{-(m+0.01)}M^{(m-0.99)}+\sum\limits_{i=\lceil Mp^{-1}\rceil}^{n-1}Mp^{-m}i^{-0.99}n^{-(m-0.99)})\\
             \leq&1+ O(1)(np^{\frac{m}{m-1}})^{1-m}M.
         \end{align*}
         Then $\p(GPM_n\text{ is connected})\geq 1- O(1)(np^{\frac{m}{m-1}})^{1-m}M \rightarrow 1$ when $np^{\frac{m}{m-1}}\rightarrow \infty$.

    \subsection{Disconnected regime}\label{disconnectss}
    We prove that with high probability $GPM_n$ has isolated vertices when $np^{\frac{m}{m-1}}$ is sufficiently small.
    Take $M=\lceil an\rceil$ where $a$ is a constant. For a small constant $c>0$ and $n_1<n_2$, define the event 
    \begin{align*}
        A_{n_1,n_2}=\{\exists \ n_1\leq i\leq n_2, L(i)\notin [(2+\delta)mpi-(pi)^c,(2+\delta)mpi+(pi)^c-m]\}.
    \end{align*}
    Then by Theorem \ref{thm2.10}, $\p(A_{M,n})=o(\exp(-(pM)^c))$. 
    Define 
    \begin{align*}
        &\tau:=\inf\{i\geq M,  i \text{~is~isolated~in~}GPM_i\},\\
        &X:=A_{M,n}^c\cap \{\tau>n\},\\
        &X_i:=A_{M,n}^c\cap \{\tau=i, \V i \text{~is~not~isolated~in~}GPM_n.\}.
    \end{align*}
    Then 
    \begin{equation}    
        \p(GPM_n \text{~is~connected})\leq \p(A)+\p(X)+\sum\limits_{i=M}^n\p(X_i).\label{eq:disconnect-1}
    \end{equation}

    Since we have
    \begin{align*}
        \p(A_{M,i}^c \cap \{\V{i} \text{ is isolate in }GPM_{i}\}\mid \mathcal{F}_{i-1})\geq \left(\frac{1+\delta}{(2+\delta)mpi+(pi)^c}\right)^m.
    \end{align*} 
    Which implies
    \begin{equation*}
        \p(X)\leq \prod\limits_{i=M}^{n}\left(1-\left(\frac{1+\delta}{(2+\delta)mpi+(pi)^c}\right)^m\right)\leq \exp(-(3m)^{-m}(1-a)n^{1-m}p^{-m}).
    \end{equation*}
    And for $j>i$, we have
    \begin{align*}
        \p(A_{M,j}^c\cap \{(\V j,\V{i})\in GPM_{j}|\mathcal{F}_{j-1}\})\leq \frac{mp\wdb{j}{i}}{(2+\delta)mpj-(pj)^c}.
    \end{align*} Notice that $\wdb{i}{j}=m(2+\delta)$ for the isolated vertex $\V i$. We have
    \begin{align*}
        \p(X_i)\leq \p(\tau=i)\left(1-\prod_{j=i+1}^{n}\left(1-\frac{(2+\delta)m^2p}{(2+\delta)mpj-(pj)^c}\right)\right) \leq \p(\tau=i)2m\log(i^{-1}n),
    \end{align*}
    and so 
    \begin{align*}
        \sum\limits_{i=M}^{n}\p(X_i)\leq -2m\log(a)\sum\limits_{i=M}^{n}\p(\tau=i)\leq 2m\frac{1-a}{a}(1-\p(X)).
    \end{align*}
    Plugging these bounds into \eqref{eq:disconnect-1}, we have
    \begin{align}
        &\p(GPM_n\text{ is connected})\leq o(\exp(-(apn)^c))+1-(1-\p(X))\left(1-2m\frac{1-a}{a}\right)\label{eq:disconnect-result-1}\\
        \leq&  o(\exp(-(apn)^c))+\exp(-(3m)^{-m}(1-a)n^{1-m}p^{-m})+2m\frac{1-a}{a}.\notag
    \end{align}    
   Then for any constant $\varepsilon>0$, there exist constants $\varepsilon',C>0$ such that when $pn>C$, $p^{\frac{m}{m-1}}n<\varepsilon'$, taking $a$ close to $1$, we have
   \begin{equation}   
       \p(GPM_n\text{ is connected})\leq \varepsilon.\label{eq:6.4}
   \end{equation}
   When $pn\leq C$, it is known that when $n>\text{e}^{Mpn}$ for some large constant $M$, with high probability the random geometric graph $RGG_{n}$ itself has isolated vertices and the network $GPM_n$ cannot be connected. If $n\leq M\text{e}^C$, then $np^{\frac{m}{m-1}}\geq (pn)^{\frac{m}{m-1}}\text{e}^{-\frac{Mpn}{m-1}}\geq \text{e}^{-CM}$. Combining with \eqref{eq:6.4}, we have
   \begin{equation*}
       \lim_{x\rightarrow 0}f_2(x)=0.
   \end{equation*}

   In addition, consider the first vertex $V_1$, if all other vertices are not in its detection area, then $\V 1$ must be isolated in $GPM_n$, which gives
   \begin{equation}
       \p(GPM_n\text{ is connected})\leq 1-\exp(-pn).\label{eq:disconnect-result-2}
   \end{equation}
   Combining \eqref{eq:disconnect-result-1} and \eqref{eq:disconnect-result-2} gives $f_2(x)<1$.

\section{Diameter}\label{sectiondia}
    In this section, we will prove Theorem \ref{main-diameter}. The proof is similar to the proof of the diameter in the PAM, which was established in \cite{Dommers2010}.
    \begin{lemma}\label{diamup}
        There is a constant $C_2$ such that almost surely, 
        \begin{align*}
            \limsup\limits_{n\rightarrow \infty}\frac{diam(GPM_n)}{\log(n)}\leq C_2.
        \end{align*}
    \end{lemma}
    \begin{proof}

    Take $M=Cp^{-\frac{m}{m-1}}$, suppose $c,\varepsilon$ are small constants, define the events
    \begin{align*}
        &A_M=\{GPM_M\text{ is not connected}\},\\
        &A_n=\{ \V n\text{ is isolated}\}\cup \{L(n)\notin [(2+\delta-\varepsilon)mpn,(2+\delta+\varepsilon)mpn]\}\cup A_{n-1}.
    \end{align*}
    Then by Theorems \ref{main-connectivity} and \ref{thm2.10}, $\lim\limits_{n\rightarrow \infty}\p(A_n)=o(1)$ when $C\rightarrow \infty$. Define 
    \begin{align*}
        X_n=\mathbf{1}_{A_n^c}\sum\limits_{i=1}^{n}\wdb{i}{n}\exp(\theta p_n(\V i,\gV{M})).
    \end{align*}
   
    If $p_{n+1}(\V{n+1},\gV{M})\geq a+1$, suppose that the first neighbor of $\V{n+1}$ is $\V i$, then $p_{n}(\V{i},\gV{M})\geq a$. As a result, on $A_{n+1}^c$, 
    \begin{align*}
        \p(p_{n+1}(\V{n+1},\gV{M})\geq a+1\mid \mathcal{F}_n)\leq\sum\limits_{\substack{1\leq i \leq n, \\p_n(i,\gV{M})\geq a}}\frac{\wdb{i}{n}}{(2+\delta-\varepsilon)mn}.
    \end{align*}
     Since
    \begin{align*}
        &\e((\wdb{i}{n+1}-\wdb{i}{n})\mathbf{1}_{A_{n+1}^c}\mid \mathcal{F}_n)\leq \frac{\wdb{i}{n}\mathbf{1}_{A_n^c}}{(2+\delta-\varepsilon)n},\\
        &p_{n+1}(x,y)\leq p_n(x,y) \quad \forall \ x,y\in \gV{n}.
    \end{align*}
    Then 
    \begin{align*}
        &\e(X_{n+1}-X_n\mid \mathcal{F}_n)\\
        \leq & \e(\sum\limits_{i=1}^n\mathbf{1}_{A_{n+1}^c}(\wdb{i}{n+1}-\wdb{i}{n})\exp(\theta p(\V i,\gV{M}))+\mathbf{1}_{A_{n+1}^c}W_{n+1}(n+1)\exp(\theta p_{n+1}(\V{n+1},\gV{M})))\\
        \leq & \sum\limits_{i=1}^n\frac{\wdb{i}{n}\exp(\theta p(\V i,\gV{M}))}{(2+\delta-\varepsilon)n}+\left(m(1+\delta)+\frac{m(1+\delta)}{(2+\delta-\varepsilon)mn}\right)\sum\limits_{a=0}^\infty \sum\limits_{\substack{1\leq i\leq n,\\ p_n(i,\gV{M})=a}}\frac{\wdb{i}{n}\exp(\theta(a+1))}{(2+\delta-\varepsilon)mn}\\
        \leq& \frac{1+\text{e}^{\theta}(1+\delta+\varepsilon)}{(2+\delta-\varepsilon)n}X_n.
    \end{align*}
    Then we have
    \begin{align*}
        \e(X_n)\leq n^{\frac{1+\text{e}^{\theta}(1+\delta+\varepsilon)}{(2+\delta-\varepsilon)}}.
    \end{align*}
    As a result, on $A_n^c$, 
    \begin{align*}
        \p(\exists\  1\leq i\leq n, p(\V i,\gV{M})\geq a)\leq \text{e}^{-a\theta}\e(X_n)\leq \exp\left(\frac{1+\text{e}^{\theta}(1+\delta+\varepsilon)}{(2+\delta-\varepsilon)}\log(n)-a\theta\right).
    \end{align*}
    Take $\theta=1, a>\frac{1+e(1+\delta+\varepsilon)}{(2+\delta-\varepsilon)}\log(n)$ to complete the proof of the lemma.
    \end{proof}
    \begin{lemma}\label{diamdown}
        There is a constant $C_1$ such that almost surely, 
        \begin{align*}
            \liminf\limits_{n\rightarrow \infty}\frac{diam(GPM_n)}{\log(n)}\geq C_1.
        \end{align*}
    \end{lemma}
    \begin{proof}
The main idea of the proof is to estimate the number of "short paths" in $GPM_n$. If the diameter of the network is less than $C_1\log(n)$, then for each vertex there exists a path of length at most $C_1\log(n)$ connecting it to one of the first $M$ vertices. Therefore, we only need to prove that the expected number of such "short paths" is $o(n)$. We cannot use Theorem \ref{thm3.1} directly because $C_1\log(n)$ is not a constant. However, we can apply the same method here.

Let $x_0 \leq M$, $x_1, x_2, \dots, x_t \geq M$ be distinct integers, and define $a_i = \min\{x_{i-1}, x_i\}$, $b_i = \max\{x_{i-1}, x_i\}$. Consider the event $G_{n_1,n_2}$ in \eqref{eq:edgeprob-lnevent} and $X_i$ in \eqref{edgevariable}. Note that if we take $M \geq \log(n)^3$ and $t \geq M$, then $k = C_1\log(n) \leq O(t^{\frac{1}{3}})$, and \eqref{edgeincrease1} remains valid for $t \notin \cup_{i=1}^k\{x_i\}$. In addition, \eqref{edgeincrease2} holds for other values of $t$. Therefore, we obtain
    \begin{align*}
        \e(X_n\mathbf{1}_{G_{M,n}}\mid \mathcal{F}_M)= O(X_M\exp(CM^{-c}k)),
    \end{align*}
    where $C$ is a large constant. Since the network formed by all $(\V{x_{i-1}},\V{x_i})$ is a line and $X_i\geq M$ for $i\geq 1$, therefore $\p(A\mid \mathcal{F}_M)=p^{k}$ and 
    \begin{align*}
        &\p((x_{i-1},x_i)\in GPM_n, \ \forall \ 1\leq i\leq k,G_{M,n})\\
        \leq &\exp(CM^{-c}k)M\prod_{i=1}^{k}\frac{(1+\delta)m+1}{(2+\delta)m}\min\{x_{i-1},x_i)^{-\frac{1}{2+\delta}}\max\{x_{i-1},x_i)^{-\frac{1+\delta}{2+\delta}}.
    \end{align*}
     Denote $l_t:=\#\{x_0 \leq M,x_1,x_2\dots, x_t> M\mid  (\V{x_i},\V{x_{i+1}})\in GPM_n ,\ \forall\ 0\leq i \leq k-1\}$, then 
         \begin{align*}
             \e(l_k)\leq \exp(O(k))M\sum\limits_{1\leq x_0\leq M, M+1\leq x_i\leq n}\prod\limits_{i=0}^{k-1}\max(x_{i},x_{i+1})^{-\frac{1+\delta}{2+\delta}}\min(x_{i},x_{i+1})^{-\frac{1}{2+\delta}}.
         \end{align*}
         By Lemma 2.4 in \cite{Dommers2010}, there is a constant $C_\delta>0$ satisfying
         \begin{align*}
         \e(l_k)\leq C_{\delta}^{k}n^{\frac{1}{2}}.    
         \end{align*}
         then for small constant $C_1$ we have
         \begin{align*}
         \sum\limits_{t=1}^{C_1\log(n)}\e(l_t)\leq n^{\frac{3}{4}},
         \end{align*}
         which completes the proof of the lemma.
\end{proof}
Combine Lemma \ref{diamup} and Lemma \ref{diamdown} to finish the proof of Theorem \ref{main-diameter}.

\section*{Acknowledgments} The authors thank Jian Ding for introducing this problem and providing many insightful discussions. The authors thank Deka Prabhanka and Zhangsong Li for many useful suggestions that greatly improve this manuscript. Both authors are supported by the National Key R\&D program of China (No.2023YFA1010103). This work is supported by Beijing Natural Science Foundation (No.QY25084).

\bibliographystyle{amsplain}

\small
\begin{thebibliography}{99}

\bibitem{Albert1999}
R{\'e}ka Albert, Hawoong Jeong, and Albert-L{\'a}szl{\'o} Barab{\'a}si.
\newblock Diameter of the world-wide web.
\newblock {\em Nature}, 401(6749):130--131, 1999.

\bibitem{BARABASI1999173}
Albert-László Barabási, Réka Albert, and Hawoong Jeong.
\newblock Mean-field theory for scale-free random networks.
\newblock {\em Physica A: Statistical Mechanics and its Applications}, 272(1):173--187, 1999.

\bibitem{Dommers2010}
Sander Dommers, Remco {van der Hofstad}, and Gerard Hooghiemstra.
\newblock Diameters in preferential attachment models.
\newblock {\em Journal of Statistical Physics}, 139(1):72--107, Apr 2010.

\bibitem{EGGEMANN2011953}
N.~Eggemann and S.D. Noble.
\newblock The clustering coefficient of a scale-free random graph.
\newblock {\em Discrete Applied Mathematics}, 159(10):953--965, 2011.

\bibitem{Faloutsos1999}
Michalis Faloutsos, Petros Faloutsos, and Christos Faloutsos.
\newblock On power-law relationships of the internet topology.
\newblock {\em SIGCOMM Comput. Commun. Rev.}, 29(4):251–262, August 1999.

\bibitem{Flaxman01012006}
Abraham~D. Flaxman, Alan~M. Frieze, and Juan~Vera and.
\newblock A geometric preferential attachment model of networks.
\newblock {\em Internet Mathematics}, 3(2):187--205, 2006.

\bibitem{Flaxman01012007}
Abraham~D. Flaxman, Alan~M. Frieze, and Juan~Vera and.
\newblock A geometric preferential attachment model of networks ii.
\newblock {\em Internet Mathematics}, 4(1):87--111, 2007.

\bibitem{Girvan2002}
Michelle Girvan and M.~E.~J. Newman.
\newblock Community structure in social and biological networks.
\newblock {\em Proceedings of the National Academy of Sciences}, 99(12):7821--7826, June 2002.

\bibitem{Jacob2015}
Emmanuel Jacob and Peter M{\"o}rters.
\newblock {Spatial preferential attachment networks: Power laws and clustering coefficients}.
\newblock {\em The Annals of Applied Probability}, 25(2):632 -- 662, 2015.

\bibitem{Jeong2000}
H.~Jeong, B.~Tombor, R.~Albert, Z.~N. Oltvai, and A.-L. Barabási.
\newblock The large-scale organization of metabolic networks.
\newblock {\em Nature}, 407(6804):651--654, October 2000.

\bibitem{Jordan2013}
Jonathan Jordan.
\newblock {Geometric preferential attachment in non-uniform metric spaces}.
\newblock {\em Electronic Journal of Probability}, 18(none):1 -- 15, 2013.

\bibitem{Watts1998}
D.~Watts and S.~Strogatz.
\newblock Collective dynamics of ‘small-world’ networks.
\newblock {\em Nature}, 393(6684):440--442, jun 1998.

\bibitem{Zuev2015}
Konstantin Zuev, Marián Boguñá, Ginestra Bianconi, and Dmitri Krioukov.
\newblock Emergence of soft communities from geometric preferential attachment.
\newblock {\em Scientific Reports}, 5(1):9421, apr 2015.

\end{thebibliography}

\end{document}